\documentclass[a4paper,12pt]{amsart}

\usepackage[T1]{fontenc}
\usepackage{inputenc}

\usepackage[english]{babel}
\usepackage{amsmath}
\usepackage{latexsym}
\usepackage{marvosym}
\usepackage{amsfonts}
\usepackage{amsthm}
\usepackage{float}
\usepackage{color}
\usepackage{epsfig}
\usepackage{graphicx}

\usepackage{enumerate}

\usepackage{stmaryrd}

\usepackage{tikz}
\usetikzlibrary{trees}
\usepackage{verbatim}
\usetikzlibrary{shapes,arrows}
\usepackage{graphicx}

\usepackage{color}

\newcommand\N{\mathbb{N}}

\newcommand\E{\mathbb{E}}
\newcommand\F{\mathbb{F}}
\newcommand\R{\mathbb{R}}

\newcommand\T{\mathbb{T}}
\newcommand\G{\mathbb{G}}

\newcommand{\veps}{\varepsilon}

\newcommand\calF{\mathcal{F}}
\newcommand\calE{\mathcal{E}}
\newcommand\calG{\mathcal{G}}
\newcommand\calB{\mathcal{B}}
\newcommand\calO{\mathcal{O}}
\newcommand\calN{\mathcal{N}}
\newcommand\calC{\mathcal{C}}
\newcommand\calV{\mathcal{V}}
\newcommand\calA{\mathcal{A}}
\newcommand\calW{\mathcal{W}}
\newcommand\calR{\mathcal{R}}
\newcommand\calT{\mathcal{T}}
\newcommand\calH{\mathcal{H}}
\newcommand\calK{\mathcal{K}}

\newcommand\wh{\widehat}
\newcommand\wt{\widetilde}
\newcommand\Var{\text{Var}}

\newcommand{\as}{\hspace{20pt} \text{a.s.}}

\renewcommand{\(}{\left(}
\renewcommand{\)}{\right)}

\def\build#1_#2^#3{\mathrel{\mathop{\kern 0pt#1}\limits_{#2}^{#3}}}
\def\liml{\build{\longrightarrow}_{}^{{\mbox{$\mathcal L$}}}}

\def\limnorm{\build{\longrightarrow}_{}^{{\mbox{$L^1$}}}}

\def\H#1{\textup{\textbf{(H.\ref{#1})}}}

\numberwithin{equation}{section}
\theoremstyle{plain}
\newtheorem{TD}{Theorem-Definition}[section]

\newtheorem{Prop}[TD]{Proposition}
\newtheorem{Theo}[TD]{Theorem}

\newtheorem{Lem}[TD]{Lemma}
\newtheorem{Rem}[TD]{Remark}

\begin{document}

\title[Asymptotic analysis for RCBAR processes]
{Asymptotic results for random coefficient bifurcating autoregressive processes}
\author{Vassili Blandin}
\dedicatory{\normalsize Universit\'e Bordeaux 1}
\address{Universit\'e Bordeaux 1, Institut de Math\'ematiques de Bordeaux, UMR CNRS 5251, and INRIA Bordeaux, team ALEA,
	351 cours de la lib\'eration, 33405 Talence cedex, France.}

\email{vassili.blandin@math.u-bordeaux1.fr}
\keywords{bifurcating autoregressive process; random coefficient; weighted least squares; martingale; almost sure convergence; central limit theorem}

\subjclass[2010]{Primary 60F15; Secondary 60F05, 60G42}

\begin{abstract}
The purpose of this paper is to study the asymptotic behavior of the weighted least squares estimators of the unknown parameters of random coefficient bifurcating autoregressive processes. Under suitable assumptions on the immigration and the inheritance, we establish the almost sure convergence of our estimators, as well as a quadratic strong law and central limit theorems. Our study mostly relies on limit theorems for vector-valued martingales.
\end{abstract}

\maketitle


\section{Introduction}


In this paper, we will study random coefficient bifurcating autoregressive processes (RCBAR). Those processes are an adaptation of random coefficient autoregressive processes (RCAR) to binary tree structured data. We can also see those processes as the combination of RCAR processes and bifurcating autoregressive processes (BAR). RCAR processes have been first studied by Nicholls and Quinn \cite{NichollsQuinn1,NichollsQuinn2} while BAR processes have been first investigated by Cowan and Staudte \cite{CowanStaudte}. Both inherited and environmental effects are taken into consideration in RCBAR processes in order to explain the evolution of the characteristic under study. The binary tree structure could lead us to take cell division as an example.\\

More precisely, the first-order RCBAR process is defined as follows. The initial cell is labelled $1$ and the offspring of the cell labelled $n$ are labelled $2n$ and $2n+1$. Denote by $X_n$ the characteristic of individual $n$. Then, the first-order RCBAR process is given, for all $n\geq1$, by
\begin{equation*}
\begin{cases}
X_{2n} &= a_n X_n + \veps_{2n}\\
X_{2n+1} &= b_n X_n + \veps_{2n+1}
\end{cases}
\end{equation*}
The environmental effect is given by the driven noise sequence $(\veps_{2n},\veps_{2n+1})_{n\geq1}$ while the inherited effect is given by the random coefficient sequence $(a_n,b_n)_{n\geq1}$. The cell division example leads us to consider that $\veps_{2n}$ and $\veps_{2n+1}$ are correlated since the environmental effect on two sister cells can reasonably be seen as correlated.\\

This study is inspired by experiments on the single celled organism \textit{Escherichia coli}, see Stewart et al.~\cite{Stewart} or Guyon et al.~\cite{Guyonbio}, which reproduces by dividing itself into two poles, one being called the new pole, the other being called the old pole. Experimental data seems to show that some variables among cell lines, such as the life span of the cells, does not evolve in the same way whether it is the new or the old pole. The difference in the evolution leads us to consider an asymmetric RCBAR. Considering a RCBAR process instead of a BAR process allows us to assume that the inherited effect is no more deterministic, as randomness often appears in nature. Moreover, we can consider both deterministic and random inherited effects since we also allow the random variables modeling the inherited effect to be deterministic, making this study usable for RCBAR as well as BAR.\\

This paper, which is an adaptation of \cite{Vassili} to RCBAR processes, intends to study the asymptotic behavior of the weighted least squares (WLS) estimators of first-order RCBAR processes using a martingale approach. This martingale approach has been first proposed by Bercu et al.~\cite{BercuBDSAGP} and de Saporta et al.~\cite{BDSAGPMarsalle} for BAR processes. The WLS estimation of parameters branching processes was previously investigated by Wei and Winnicki \cite{WeiWinnicki} and Winnicki \cite{Winnicki}. We will make use several times of the strong law of large numbers \cite{Duflo} as well as the central limit theorem \cite{Duflo,HallHeyde} for martingales, in order to investigate the asymptotic behavior of the WLS estimators. Those theorems have been previously used by Basawa and Zhou \cite{BasawaZhou,ZhouBasawa,ZhouBasawa2}.\\

Several approaches appeared for BAR processes, and we tried not to set aside any of them. Thus, we took into account the classical BAR studies as seen in Huggins and Basawa \cite{HugginsBasawa99,HugginsBasawa2000} and Huggins and Staudte \cite{HugginsStaudte} who studied the evolution of cell diameters and lifetimes, and also the bifurcating Markov chain model introduced by Guyon \cite{Guyon} and used in Delmas and Marsalle \cite{DelmasMarsalle}. Still, we did not forget to have a look to the analogy with the Galton-Watson processes as studied in Delmas and Marsalle \cite{DelmasMarsalle} and Heyde and Seneta \cite{HeydeSeneta}. Several methods have also been used for parameter estimation in RCAR processes. Koul and Schick \cite{KoulSchick} used an M-estimator while Aue et al. \cite{AueHorvathSteinebach} preferred a quasi-maximum likelihood approach. Schick \cite{Schick} introduced a new class of estimator that Vanecek \cite{Vanecek} used in his work. Hwang et al. \cite{HwangBasawaKim} also tackled the critical case where the environmental effect follows a Rademacher distribution.\\

The paper is organized as follows. Section 2 allows us to explain more precisely the model in which we are interested in, then Section 3 formulates the WLS estimators of the unknown parameters we will study. Section 4 permits us to introduce the martingale point of view of this paper. The main results are collected in Section 5, those results concern the asymptotic behavior of our WLS estimators, to be more accurate, we will establish the almost sure convergence, the quadratic strong law and the asymptotic normality of our estimators. Finally, the other sections gathers the proofs of our main results, except the last section which illustrates our results with a small simulation study.


\section{Random coefficient bifurcating autoregressive processes}\label{section2}


Consider the first-order RCBAR process given, for all $n\geq1$, by
\begin{equation}
\left\{
\begin{aligned}\label{defsyst}
&X_{2n} &= a_n X_n & + \veps_{2n}\\
&X_{2n+1} &= b_n X_n & + \veps_{2n+1}
\end{aligned}
\right.
\end{equation}
where
the initial state $X_1$ is the ancestor of the process and $(\veps_{2n},\veps_{2n+1})$ stands for the driven noise of the process. In all the sequel, we shall assume that $\E[X_1^2]<\infty$. We also assume that  both $(a_n,b_n)_{n\geq1}$ and $(\veps_{2n},\veps_{2n+1})_{n\geq1}$ are i.i.d., and that those two sequences are independent. One can see the RCBAR process given by \eqref{defsyst} as a first-order random coefficient autoregressive process on a binary tree, where each node 
represents an individual, node 1 being the original ancestor. For all $n\geq1$, denote the $n$-th generation by $\G_n = \{2^n, 2^{n}+1, \hdots, 2^{n+1}-1\}$. In particular, $\G_0 = \{1\}$ is the initial generation and $\G_1 = \{2,3\}$ is the first generation of offspring 
from the first ancestor. Recall that the two 
offspring of individual $n$ are labelled $2n$ and $2n+1$, or conversely, the mother of individual $n$ is $[n/2]$ where $[x]$ stands for 
the largest integer less than or equal to $x$. Finally 
denote by
$$\T_n = \bigcup_{k=0}^n \G_n$$

\noindent the sub-tree of all individuals from the original individual up to the $n$-th generation. On can observe that the cardinality $|\G_n|$ 
of $\G_n$ is $2^n$ while that of $\T_n$ is $|\T_n| = 2^{n+1}-1$.\\

\begin{figure}[h!]
\hspace{-2cm}
\includegraphics[scale=0.7]{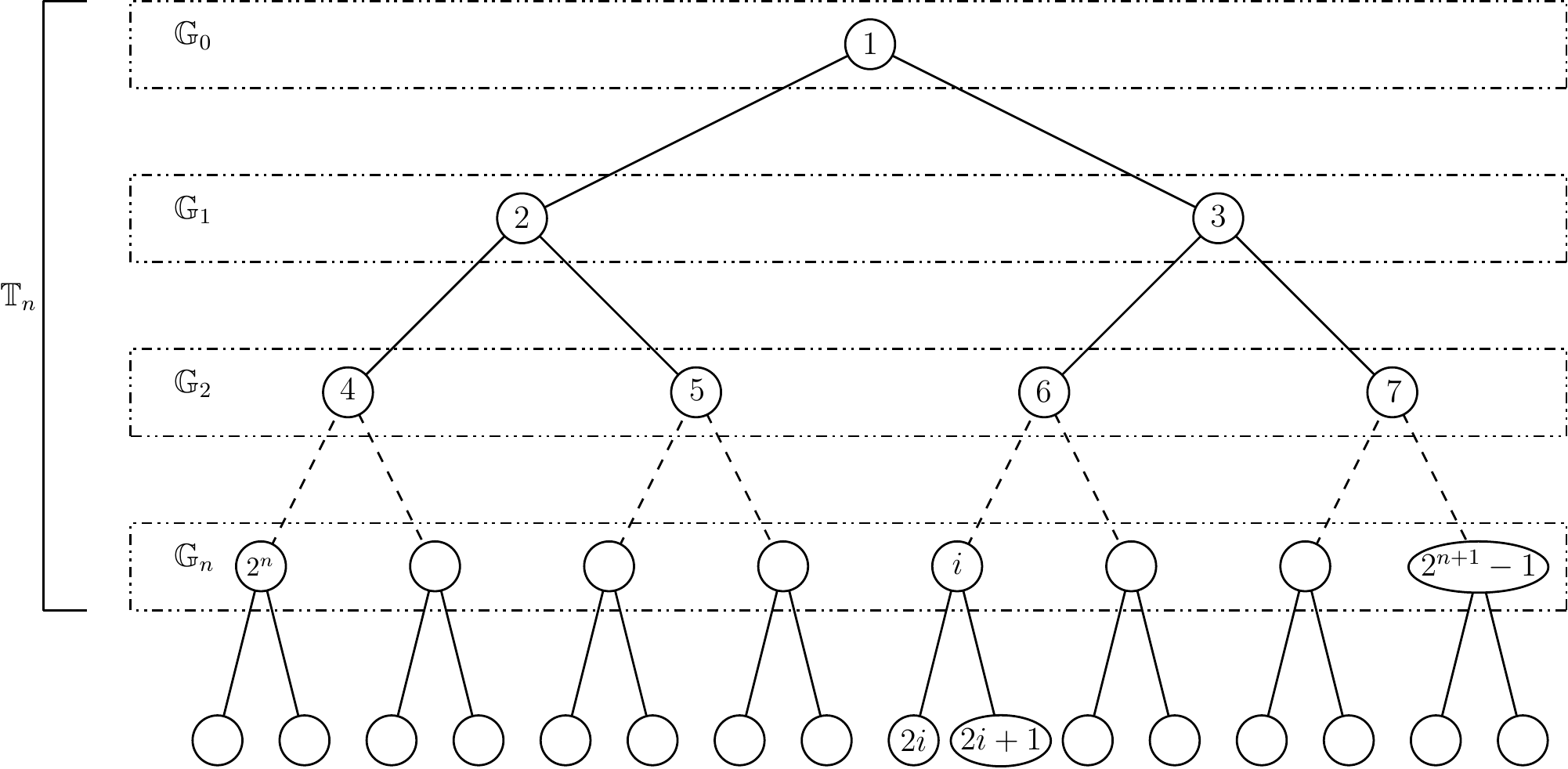}
\caption{The tree associated with the RCBAR}
\end{figure}


\section{Weighted least-squares estimation}


Denote by $\F = (\calF_n)_{n\geq0}$ the natural filtration associated with the first-order RCBAR process, which means that 
$\calF_n$ is the $\sigma$-algebra generated by all individuals up to the $n$-th generation, in other words $\calF_n=\sigma\{X_k,k\in\T_n\}$. We will assume in all the sequel that, for all $n\geq0$ and for all $k\in \G_n$,
\begin{equation}\label{espcond}
\begin{cases}
\E[a_k|\calF_n]=a \as\\
\E[b_k|\calF_n]=b \as\\
\E[\veps_{2k}|\calF_n] = c \hspace{20pt} \text{a.s.}\\
\E[\veps_{2k+1}|\calF_n] = d \hspace{20pt} \text{a.s.}
\end{cases}
\end{equation}
Consequently, we deduce from \eqref{defsyst} and \eqref{espcond} that, for all $n\geq0$ and for all $k\in \G_n$,
\begin{equation}
\label{linsyst}
\begin{cases}
X_{2k} &= a X_k + c + V_{2k},\\
X_{2k+1} &= b X_k + d + V_{2k+1},
\end{cases}
\end{equation}

\noindent where, $V_{2k} = X_{2k} - \E[X_{2k}|\calF_{n}]$ and $V_{2k+1} = X_{2k+1} - \E[X_{2k+1}|\calF_{n}]$. Therefore, the two relations given by \eqref{linsyst} can be rewritten in a classic autoregressive form
\begin{equation}
 \label{matsyst}
 \chi_n = \theta^t\Phi_n + W_n
\end{equation}
where
$$
\begin{array}{ccccc}
 \chi_n = \begin{pmatrix} X_{2n}\\ X_{2n+1} \end{pmatrix},
  & & \Phi_n = \begin{pmatrix} X_{n}\\ 1 \end{pmatrix},
  & & W_n = \begin{pmatrix} V_{2n}\\ V_{2n+1} \end{pmatrix},
\end{array}
$$
and the matrix parameter
$$\theta = \begin{pmatrix} a & b \\ c & d \end{pmatrix}.$$

\noindent Our goal is to estimate $\theta$ from the observation of all individuals up to $\T_n$. We propose 
to make use of the WLS estimator $\wh{\theta}_n$ of $\theta$ which minimizes
$$\Delta_n(\theta) = \frac12\sum_{k \in \T_{n-1}} \frac1{c_k}\|\chi_k-\theta^t \Phi_k\|^2$$

\noindent where the choice of the weighting sequence $(c_n)_{n\geq1}$ is crucial. We shall choose $c_n=1+X_n^2$ and we will go back to this suitable choice in Section \ref{martappro}. Consequently, we obviously have for all $n\geq1$
\begin{equation}
\label{eqest}
\wh{\theta}_n = S_{n-1}^{-1} \sum_{k\in\T_{n-1}} \frac1{c_k}\Phi_k\chi_k^t, \hspace{20pt} \text{where} \hspace{20pt} S_n = \sum_{k\in\T_n} \frac1{c_k} \Phi_k \Phi_k^t.
\end{equation}

\noindent In order to avoid useless invertibility assumption, we shall assume, without loss of generality, that for all $n\geq0$, $S_n$ 
is invertible. Otherwise, we only have to add the identity matrix of order 2, $I_2$ to $S_n$. In all what follows, we shall make a slight 
abuse of notation by identifying $\theta$ as well as $\wh\theta_n$ to
$$\begin{array}{ccccc}
 \text{vec}(\theta) = \begin{pmatrix} a \\ c \\  b \\ d \end{pmatrix} & & \text{and} & & 
 \text{vec}(\wh \theta_n) = \begin{pmatrix} \wh{a}_n \\ \wh{c}_n \\ \wh{b}_n \\ \wh{d}_n \end{pmatrix}.
\end{array}$$

\noindent Therefore, we deduce from \eqref{eqest} that
\begin{align}
 \wh{\theta}_n = \Sigma_{n-1}^{-1} \sum_{k\in\T_n-1} \frac1{c_k}\text{vec}(\Phi_k\chi_k^t)
  = \Sigma_{n-1}^{-1} \sum_{k\in\T_n-1} \frac1{c_k} \begin{pmatrix} X_kX_{2k} \\ X_{2k} \\ X_kX_{2k+1} \\ X_{2k+1} \end{pmatrix} \nonumber
\end{align}

\noindent where $\Sigma_n = I_2 \otimes S_n$ and $\otimes$ stands for the standard Kronecker product. Consequently, \eqref{matsyst} yields to
\begin{align}
 \wh{\theta}_n -\theta &= \Sigma_{n-1}^{-1} \sum_{k\in\T_{n-1}} \frac1{c_k}\text{vec}(\Phi_k W_k^t), \nonumber \\
  \label{diffest} &= \Sigma_{n-1}^{-1} \sum_{k\in\T_{n-1}} \frac1{c_k} \begin{pmatrix} X_kV_{2k} \\ V_{2k} \\ X_kV_{2k+1} 
									    \\ V_{2k+1} \end{pmatrix}.
\end{align}

\noindent In all the sequel, we shall make use of the following moment hypotheses.

\begin{enumerate}[\bf{({H}.}1)]
\item\label{H1} For all $k\geq1$, 
\begin{equation*}
\E[a_k^2] < 1 \hspace{20pt}  \text{and} \hspace{20pt} \E[b_k^2] < 1.
\end{equation*}
 \item \label{H2} For all $n\geq0$ and for all $k\in \G_n$
$$\begin{array}{ccccccc}
  \Var[a_k|\calF_n] = \sigma_a^2 \geq 0 & \text{and} &  \Var[b_k|\calF_n] = \sigma_b^2 \geq 0 & & \text{a.s.}
\end{array}$$
$$\begin{array}{ccccccccccc}
  \Var[\veps_{2k}|\calF_n] = \sigma_c^2 > 0 & \text{and} &  \Var[\veps_{2k+1}|\calF_n] = \sigma_d^2 > 0 & & \text{a.s.}
\end{array}$$
 \item \label{H3} For all $n\geq0$ and for all $k,l\in\G_{n+1}$, if $[k/2]\neq[l/2]$, $\veps_k$ and $\veps_l$ are conditionally independent given $\calF_n$ and for all $k,l\in\G_{n}$, if $k\neq l$, $(a_k,b_k)$ and $(a_l,b_l)$ are conditionally independent given $\calF_n$. While otherwise, it exists $\rho_{cd}^2 < \sigma_c^2 \sigma_d^2$ and $\rho_{ab}^2\leq\sigma_a^2\sigma_b^2$ such that, for all $k\in\G_n$
$$\E[(\veps_{2k}-c)(\veps_{2k+1}-d)|\calF_n] = \rho_{cd} \hspace{20pt} \text{ a.s.}$$
$$\E[(a_k-a)(b_k-b)|\calF_n]=\rho_{ab} \as$$
\item \label{H4} One can find $\mu_a^4\geq\sigma_a^4$, $\mu_b^4\geq\sigma_b^4$, $\mu_c^4>\sigma_c^4$ and $\mu_d^4>\sigma_d^4$ such that, for all $n\geq0$ and for all $k\in \G_n$
$$\begin{array}{ccccccc}
  \E\left[\left(a_{k}-a\right)^4|\calF_n\right] = \mu_a^4 & & \text{and} & & \E\left[\left(b_k-b\right)^4|\calF_n\right] = \mu_c^4 & & \text{a.s.}
\end{array}$$
$$\begin{array}{ccccccc}
  \E\left[\left(\veps_{2k}-c\right)^4|\calF_n\right] = \mu_c^4 & & \text{and} & & \E\left[\left(\veps_{2k+1}-d\right)^4|\calF_n\right] = \mu_d^4 & & \text{a.s.}
\end{array}$$
$$\E[\veps_{2k}^4]>\E[\veps_{2k}^2]^2 \hspace{20pt} \text{ and } \hspace{20pt} \E[\veps_{2k+1}^4]>\E[\veps_{2k+1}^2]^2.$$
In addition, it exists $\nu_{ab}^2 \geq \rho_{ac}^2$ and $\nu_{cd}^2 > \rho_{cd}^2$ such that, for all $k\in\G_n$
$$\E[(a_k-a)^2(b_k-b)^2|\calF_n] = \nu_{ab}^2 \hspace{20pt} \text{and} \hspace{20pt} \E[(\veps_{2k}-c)^2(\veps_{2k+1}-d)^2|\calF_n] = \nu_{cd}^2 \hspace{20pt} \text{ a.s.}$$
\item \label {H5}It exists $\alpha>4$ such that
$$\sup_{n\geq0}\sup_{k\in\G_n} \E[|a_k-a|^\alpha|\calF_n]<\infty, \hspace{20pt} \sup_{n\geq0}\sup_{k\in\G_n} \E[|b_k-b|^\alpha|\calF_n]<\infty \as$$
$$\sup_{n\geq0}\sup_{k\in\G_n} \E[|\veps_{2k}-c|^\alpha|\calF_n]<\infty, \hspace{20pt} \sup_{n\geq0}\sup_{k\in\G_n} \E[|\veps_{2k+1}-d|^\alpha|\calF_n]<\infty \as$$
\end{enumerate}

\vspace{25pt}

\noindent One can observe that those hypotheses allows us to consider the deterministic case where it exists some constants $a$, $b$ with $\max(|a|,|b|)<1$ such that, for all $k\geq1$, $a_k=a$ and $b_k=b$ a.s. Moreover, under assumption \H{H2}, we have for all $n\geq0$ and for all $k\in\G_n$
\begin{eqnarray}
   \E[V_{2k}^2|\calF_{n}] = \sigma_a^2X_k^2 + \sigma_c^2 & \text{ and } & \E[V_{2k+1}^2|\calF_{n}] = \sigma_b^2X_k^2 + \sigma_d^2 \text{\hspace{20pt} a.s.} \label{espV2}
  \end{eqnarray}

\noindent Consequently, if we choose $c_n=1+X_n^2$ for all $n\geq1$, we clearly have for all $k\in\G_n$
$$\begin{array}{cccc}
\E\left[\left.V_{2k}^2\right|\calF_{n}\right] \leq \max(\sigma_a^2,\sigma_c^2)c_k & \text{ and } &
\E\left[\left.V_{2k+1}^2\right|\calF_{n}\right] \leq \max(\sigma_b^2,\sigma_d^2)c_k &  \text{ a.s.}
\end{array}$$

\noindent It is exactly the reason why we have chosen this weighting sequence into \eqref{eqest}. Similar WLS estimation approach for branching processes with immigration may be found in \cite{WeiWinnicki} and \cite{Winnicki}. For all $n\geq0$ and for all $k\in\G_n$, denote $v_{2k} = V_{2k}^2 - \E[V_{2k}^2 | \calF_n]$. We deduce from \eqref{espV2} that for all $n\geq1$,  $V_{2n}^2 = \eta^t \psi_n + v_{2n}$ where $\eta$ is defined by
$$\eta = \begin{pmatrix} \sigma_a^2 \\ \sigma_c^2 \end{pmatrix} \hspace{20pt} \text{and} \hspace{20pt} \psi_n=\begin{pmatrix} X_n^2 \\ 1 \end{pmatrix}.$$
It leads us to estimate the vector of variances $\eta$ by the WLS estimator
\begin{equation} \wh \eta_n = Q_{n-1}^{-1} \sum_{k\in \T_{n-1}} \frac{1}{d_k} \wh V_{2k}^2 \psi_k, \label{esteta}  \hspace{20pt} \text{where} \hspace{20pt} Q_n = \sum_{k\in\T_n} \frac1{d_k} \psi_k \psi_k^t\end{equation}
and for all $k\in\G_n$, 
\begin{equation*}
\begin{cases}
\wh V_{2k} &= X_{2k} - \wh a_n X_k - \wh c_n,\vspace{1ex}\\
 \wh V_{2k+1} &= X_{2k+1} - \wh b_n X_k - \wh d_n.
\end{cases}
\end{equation*}
Finally the weighting sequence $(d_n)_{n\geq1}$ is given, for all $n\geq1$, by $d_n=c_n^2=(1+X_n^2)^2$. This choice is due to the fact that for all $n\geq1$ and for all $k\in\G_n$
\begin{align*}
\E[v_{2k}^2|\calF_n] &= \E[V_{2k}^4|\calF_n] - \left(\E[V_{2k}^2|\calF_n]\right)^2 \hspace{20pt} \text{ a.s.}\nonumber\\
&=(\mu_a^4 - \sigma_a^4) X_k^4 + 4 \sigma_a^2\sigma_c^2 X_k^2 +(\mu_c^4-\sigma_c^4) \as \label{Ev2k}
\end{align*}
 Consequently, as $d_n\geq1$, we clearly have for all $n\geq1$ and for all $k\in\G_n$
$$\E[v_{2k}^2|\calF_n] \leq \max(\mu_a^4 - \sigma_a^4, 2\sigma_a^2\sigma_c^2, \mu_c^4-\sigma_c^4) d_k \hspace{20pt} \text{ a.s.}$$

\noindent We have a similar WLS estimator $\wh\zeta_n$ of the vector of variances 
$$\zeta^t =\begin{pmatrix} \sigma_b^2 & \sigma_d^2 \end{pmatrix}$$ 
by replacing $\wh V_{2k}^2$ by $\wh V_{2k+1}^2$ into \eqref{esteta}. Let us remark that, for all $n\geq0$ and for all $k\in\G_n$,
\begin{equation}\label{espcroise}
\E[V_{2k}V_{2k+1}|\calF_n] = \rho_{ab} X_n^2 + \rho_{cd}.
\end{equation}
Then, for all $n\geq0$ and for all $k\in\G_n$, denote $w_{2k} = V_{2k}V_{2k+1} - \E[V_{2k}V_{2k+1} | \calF_n]$. We deduce from \eqref{espcroise} that for all $k\geq1$, $V_{2k}V_{2k+1} = \nu^t \psi_k + w_{2k}$
where $\nu$ is defined by
$$\nu = \begin{pmatrix} \rho_{ab} \\ \rho_{cd} \end{pmatrix}.$$
It leads us to estimate the vector of covariances $\nu$ by the WLS estimator
\begin{equation}\label{carre} \wh \nu_n = Q_{n-1}^{-1} \sum_{k\in\T_{n-1}} \frac1{d_k} \wh V_{2k} \wh V_{2k+1} \psi_k. \end{equation}
This choice is due to the fact that for all $n\geq1$ and for all $k\in\G_n$
$$\E[V_{2k}^2V_{2k+1}^2|\calF_n] = \nu_{ab}^2 X_k^4 + (\sigma_a^2\sigma_d^2 +4 \rho_{ab} \rho_{cd} + \sigma_b^2\sigma_c^2)X_k^2 + \nu_{cd}^2 \as$$
 Consequently, as $d_n\geq1$, we clearly have for all $n\geq1$ and for all $k\in\G_n$
\begin{align*}
\E[w_{2k}^2|\calF_n] &=(\nu_{ab}^2-\rho_{ab}^2)X_k^4 + \left(\sigma_a^2\sigma_d^2 + \sigma_b^2\sigma_c^2 + 2\rho_{ab}\rho_{cd}\right)X_k^2 + (\nu_{cd}^2-\rho_{cd}^2) \as\\
	&\leq \max\(\nu_{ab}^2,\nu_{cd}^2,\(\sigma_a^2+\sigma_c^2\)\(\sigma_b^2+\sigma_d^2\)\) d_k \hspace{20pt} \text{ a.s.}
\end{align*}


\section{A martingale approach}\label{martappro}


In order to establish all the asymptotic properties of our estimators, we shall make use of a martingale approach. For all $n\geq1$, denote
$$ M_n = \sum_{k\in\T_{n-1}} \frac1{c_k} \begin{pmatrix} X_kV_{2k} \\ V_{2k} \\ X_kV_{2k+1} \\ V_{2k+1} \end{pmatrix}. $$

\noindent We can clearly rewrite \eqref{diffest} as
\begin{equation}\wh\theta_n - \theta = \Sigma_{n-1}^{-1} M_n.\label{difftheta}\end{equation}

\noindent As in \cite{BercuBDSAGP}, we make use of the notation $M_n$ since it appears that $(M_n)_{n\geq1}$ is a martingale. This fact is a crucial point of our study and it justifies the vector notation since most of all asymptotic results for martingales were established for vector-valued martingales. Let us rewrite $M_n$ in order to emphasize its martingale quality. Let $\Psi_n = I_2 \otimes \varphi_n$ 
where $\varphi_n$ is the matrix of dimension $2\times2^n$ given by
\begin{equation*}\varphi_n = \begin{pmatrix} \displaystyle \frac{X_{2^n}}{\sqrt{c_{2^n}}} & \displaystyle \frac{X_{2^n+1}}{\sqrt{c_{2^n+1}}} & \displaystyle \hdots & 
		\displaystyle \frac{X_{2^{n+1}-1}}{\sqrt{c_{2^{n+1}-1}}} \vspace{5pt} \\
	    \displaystyle \frac1{\sqrt{c_{2^n}}} & \displaystyle \frac1{\sqrt{c_{2^n+1}}} & \displaystyle \hdots & 
		\displaystyle \frac1{\sqrt{c_{2^{n+1}-1}}}
\end{pmatrix}.\end{equation*}

\noindent It represents the individuals of the $n$-th generation which is also the collection of all $\Phi_k/\sqrt{c_k}$ where $k$ belongs to $\G_n$. Let $\xi_n$ 
be the random vector of dimension $2^n$
$$\xi_n^t = \begin{pmatrix} \displaystyle\frac{V_{2^n}}{\sqrt{c_{2^{n-1}}}} & \displaystyle\frac{V_{2^n+2}}{\sqrt{c_{2^{n-1}+1}}} & \hdots &
			   \displaystyle\frac{V_{2^{n+1}-2}}{\sqrt{c_{2^{n}-1}}} &  \displaystyle\frac{V_{2^n+1}}{\sqrt{c_{2^{n-1}}}} &
			  \displaystyle\frac{V_{2^n+3}}{\sqrt{c_{2^{n-1}+1}}} & \hdots & \displaystyle\frac{V_{2^{n+1}-1}}{\sqrt{c_{2^{n}-1}}}
	    \end{pmatrix}.
$$

\noindent The vector $\xi_n$ gathers the noise variables of $\G_n$. The special ordering separating odd and even indices has been made in \cite{BercuBDSAGP}
so that $M_n$ can be written as
$$M_n = \sum_{k=1}^n \Psi_{k-1} \xi_k.$$
Under \eqref{espcond}, we clearly have for all $n\geq0$, $\E[\xi_{n+1}|\calF_n] = 0$ a.s.~and $\Psi_n$ is $\calF_n$-measurable. In addition it is 
not hard to see that under \H{H1} to \H{H2}, $(M_n)$ is a locally square integrable vector martingale with increasing 
process given, for all $n\geq1$, by
\begin{align}
\langle M\rangle_n &= \sum_{k=0}^{n-1} \Psi_k \E[\xi_{k+1}\xi_{k+1}^t|\calF_k]\Psi_k^t =\sum_{k=0}^{n-1} L_k \hspace{20pt} \text{a.s.} \label{defcrochetM}
\end{align}

\noindent where 
\begin{equation}\label{defLk}
    L_k = \sum_{i\in\G_{k}} \frac1{c_i^2} \begin{pmatrix} P(X_i) & Q(X_i) \\ Q(X_i) & R(X_i) \end{pmatrix}\otimes\begin{pmatrix} X_i^2 & X_i \\ X_i & 1\end{pmatrix}.
\end{equation}
with
$$\begin{cases} P(X) = \sigma_a^2X^2 + \sigma_c^2,\\
Q(X) = \rho_{ab} X^2 + \rho_{cd},\\
R(X) = \sigma_b^2 X^2 + \sigma_d^2.
\end{cases}$$

\noindent One can remark that we obviously have $\langle M\rangle _n=\calO(\T_n)$ but it is necessary to establish the convergence of $\langle M\rangle _n$, properly normalized, in order to prove the asymptotic results for our RCBAR estimators $\wh\theta_n$, $\wh\eta_n$, $\wh\zeta_n$ and $\wh \nu_n$.


\section{Main results} \label{mainresults}


We have to introduce some more notations in order to state our main results. From the original process $(X_n)_{n\geq1}$, we shall define a new process $(Y_n)_{n\geq1}$ recursively defined by $Y_1=X_1$, and if $Y_n=X_k$ with $n,k\geq1$, then 
$$Y_{n+1} = X_{2k+\kappa_n}$$

\noindent where $(\kappa_n)_{n\geq1}$ is a sequence of i.i.d.~random variables with Bernoulli $\calB\left(1/2\right)$ distribution. Such a construction may be found in \cite{Guyon} for the asymptotic analysis of BAR processes. The process $(Y_n)$ gathers the values of the original process $(X_n)$ along the random branch of the binary tree $(\T_n)$ given by $(\kappa_n)$. Denote by $k_n$ the unique $k\geq1$ such that $Y_n=X_k$. Then, for all $n\geq1$, we have
\begin{equation} Y_{n+1} = \wt a_{n+1}  Y_{n} + e_{n+1} \label{defYn1}\end{equation}
where, with $k_n$ the unique number $k$ such that $Y_n = X_k$,
\begin{equation}\label{defYn2}
\wt a_{n+1} = \begin{cases}a_{k_n} \text{ if } \kappa_n=0, \\ b_{k_n} \text{ otherwise,} \end{cases} \hspace{20pt} \text{and} \hspace{20pt} e_n = \veps_{k_n}.
\end{equation}

\begin{Lem} \label{CVYn}
 Assume that \H{H1} and \H{H2} are satisfied. Then, we have
\begin{equation*}
 Y_n \liml T 
\end{equation*}
where $T$ is a positive non degenerate random variable with $\E[T^2]<\infty$.
\end{Lem}

\noindent Denote
$\calC_{b}^1(\R_+) = \Bigl\{f\in\calC^1(\R,\R)\big|\exists \gamma >0, \forall x\geq0, (|f'(x)| + |f(x)|) \leq \gamma\Bigl\}$.

\begin{Lem}\label{LFGN}
Assume that \H{H1} and 
\H{H2} are satisfied. Then, for all $f\in \calC_{b}^1(\R_+)$, we have
$$\lim_{n\to\infty} \frac1{|\T_n|} \sum_{k\in \T_n} f(X_k) = \E[f(T)] \hspace{20pt}  \text{a.s.}$$
\end{Lem}

\begin{Prop} \label{cvcrochet}
 Assume that \H{H1} to \H{H3} are satisfied. Then, we have 
\begin{equation} \label{limcrochet}
 \lim_{n\to\infty} \frac{\langle M\rangle _n}{|\T_{n-1}|} = L \hspace{20pt} \text{ a.s.}
\end{equation}
\noindent where $L$ is the positive definite matrix given by
\begin{equation*} L = 
\E\left[\frac1{(1+T^2)^2} \begin{pmatrix} P(T) & Q(T) \\ Q(T) & R(T) \end{pmatrix} \otimes \begin{pmatrix} T^2 & T \\ T & 1 \end{pmatrix} \right].\end{equation*}

\end{Prop}

\noindent Our first result deals with the almost sure convergence of our WLS estimator $\wh\theta_n$.

\begin{Theo} \label{CVpstheta}
 Assume that \H{H1} to \H{H5} satisfied. Then, $\wh\theta_n$ converges almost surely 
to $\theta$ with the rate of convergence
\begin{equation*}
\|\wh\theta_n-\theta\|^2 =  \calO\left(\frac{n}{|\T_{n-1}|}\right) \hspace{20pt} \text{ a.s.} \label{rate}
\end{equation*}
\noindent In addition, we also have the quadratic strong law
\begin{equation}\label{quadratic1}
 \lim_{n\to\infty} \frac1n \sum_{k=1}^n |\T_{k-1}|(\wh \theta_k -\theta)^t \Lambda (\wh \theta_k -\theta) = tr(\Lambda^{-1/2}L\Lambda^{-1/2}) \hspace{20pt} \text{ a.s.}\\
\end{equation}\noindent where
\begin{equation}\label{defA}
\Lambda = I_2\otimes C \hspace{20pt} \text{ and } \hspace{20pt} C = \E\left[\frac1{1+T^2}\begin{pmatrix} T^2 & T \\ T & 1 \end{pmatrix}\right].
\end{equation}

\end{Theo}

\noindent Our second result concerns the almost sure asymptotic properties of our WLS variance and covariance estimators $\wh \eta_n$, $\wh \zeta_n$ and $\wh \nu_n$. Let
\begin{equation*}
\eta_n  = Q_{n-1}^{-1} \sum_{k\in \T_{n-1}} \frac{1}{d_k} V_{2k}^2 \psi_k, \hspace{40pt}
\zeta_n = Q_{n-1}^{-1} \sum_{k\in \T_{n-1}} \frac{1}{d_k} V_{2k+1}^2 \psi_k,
\end{equation*}
\begin{equation*}
\nu_n = Q_{n-1}^{-1}\sum_{k\in\T_{n-1}}\frac1{d_k} V_{2k}V_{2k+1}\psi_k.
\end{equation*}

\begin{Theo}\label{CVpsvar}
 Assume that \H{H1} to \H{H5} are satisfied. Then, $\wh \eta_n$ and $ \wh \zeta_n$ converge almost surely to $\eta$ and 
$\zeta$ respectively. More precisely,
\begin{align}
\|\wh \eta_{n} - \eta_{n}\|  &= \calO\left(\frac{n}{|\T_{n-1}|}\right) \hspace{20pt} \text{ a.s.}\label{vitesseeta} \\
\|\wh \zeta_{n} - \zeta_{n}\| &= \calO\left(\frac{n}{|\T_{n-1}|}\right) \hspace{20pt} \text{ a.s.}\label{vitesseetad} 
\end{align}

\noindent In addition, $\wh \nu_n$ converges almost surely to $\nu$ with
\begin{equation}
\|\wh \nu_n - \nu_n\| = \calO\left(\frac n{|\T_{n-1}|}\right) \hspace{20pt} \text{ a.s.} \label{vitesserho}
\end{equation}

\end{Theo}

\begin{Rem}\label{remrate}
We also have the almost sure rates of convergence
$$\|\wh \eta_{n} - \eta\|^2 =\calO\left(\frac{n}{|\T_{n-1}|}\right),~~ \|\wh \zeta_{n} - \zeta\|^2 =\calO\left(\frac{n}{|\T_{n-1}|}\right),~~ \|\wh \nu_{n} - \nu\|^2 =\calO\left(\frac{n}{|\T_{n-1}|}\right) ~~~ a.s.$$
\end{Rem}

\noindent Our last result is devoted to the asymptotic normality of our WLS estimators $\wh\theta_n$, $\wh \eta_n$, $\wh \zeta_n$ and $\wh \nu_n$.

\begin{Theo}\label{TCL}
 Assume that \H{H1} to \H{H5} are satisfied. Then, we have the asymptotic normality
\begin{equation} \label{TCLtheta}
 \sqrt{|\T_{n-1}|}(\wh\theta_n - \theta) \liml \calN(0,\Lambda^{-1}L\Lambda^{-1}).
\end{equation}

\noindent In addition, we also have
\begin{eqnarray}
\sqrt{|\T_{n-1}|}\left(\wh \eta_{n} - \eta\right)  \liml \calN(0,D^{-1} M_{ac} D^{-1}) \label{TCLeta},\\
\sqrt{|\T_{n-1}|}\left(\wh \zeta_{n} - \zeta\right) \liml \calN(0,D^{-1} M_{bd} D^{-1}) \label{TCLetad},
\end{eqnarray}
\noindent where
$$D = \E\left[\frac1{(1+T^2)^2} \begin{pmatrix} T^4 & T^2 \\ T^2 & 1 \end{pmatrix} \right],$$
$$M_{ac} = \E\left[\frac{(\mu_a^4-\sigma_a^4)T^4+4\sigma_a^2\sigma_c^2T^2+(\mu_c^4-\sigma_c^4)}{(1+T^2)^4} \begin{pmatrix} T^4 & T^2 \\ T^2 & 1 \end{pmatrix} \right],$$
$$M_{bd} = \E\left[\frac{(\mu_b^4-\sigma_b^4)T^4+4\sigma_b^2\sigma_d^2T^2+(\mu_d^4-\sigma_d^4)}{(1+T^2)^4} \begin{pmatrix} T^4 & T^2 \\ T^2 & 1 \end{pmatrix} \right].$$
Finally,
\begin{equation} 
\sqrt{|\T_{n-1}|}  \left(\wh \nu_n - \nu\right) \liml \calN\left(0,D^{-1} H D^{-1} \right) \label{TCLrho}
\end{equation}

\noindent where 
$$H = \E\left[\frac{(\nu_{ab}^2-\rho_{ab}^2)T^4 + (\sigma_a^2 \sigma_d^2 + \sigma_b^2 \sigma_c^2+2\rho_{ab}\rho_{cd})T^2 + (\nu_{cd}^2 - \rho_{cd}^2)}{(1+T^2)^4} \begin{pmatrix} T^4 & T^2 \\ T^2 & 1 \end{pmatrix}\right].$$

\end{Theo}

\noindent The rest of the paper is dedicated to the proof of our main results.


\section{Proof of Lemma \ref{CVYn}} \label{preuve lemme CVYn}


We can reformulate \eqref{defYn1} and \eqref{defYn2} as
$$Y_n = \wt a_n \wt a_{n-1}  \hdots \wt a_2 Y_1 + \sum_{k=2}^{n-1} \wt a_n \wt a_{n-1} \hdots \wt a_{k+1} e_k 
      + e_n.$$

\noindent We already made the assumption that both $(a_n,b_n)_{n\geq1}$ and $(\veps_{2n},\veps_{2n+1})_{n\geq1}$ are i.i.d. and that those two sequences are independent. Consequently, the couples $(\wt a_k,e_k)$ and $(\wt a_{n-k+2},e_{n-k+1})$ share the same distribution. Hence, for all $n\geq2$, $Y_n$ has the same distribution than the random variable
\begin{align*}
 Z_n &= \wt a_2 \hdots \wt a_n Y_1 + \sum_{k=2}^{n-1} \wt a_2 \wt a_3 \hdots \wt a_{n-k+1} e_{n-k+2} 
    + e_2,\\
	      &= \wt a_2 \hdots \wt a_n Y_1 + \sum_{k=3}^{n} \wt a_2 \wt a_3 \hdots \wt a_{k-1} e_{k} 
    + e_2.
\end{align*}

\noindent For the sake of simplicity, we will denote 
\begin{equation}Z_n = \wt a_2 \hdots \wt a_n Y_1 + \sum_{k=2}^{n} \wt a_2 \wt a_3 \hdots \wt a_{k-1} e_{k}.\label{defZn}\end{equation}
On the first hand,
$\E[\wt a_2 \wt a_3 \hdots \wt a_n Y_1] = \E[\wt a_2]^{n-1} \E[Y_1]$
and since 
$$\left|\E[\wt a_2]\right| = \left|\frac{a+b}2\right| <1$$
this immediately leads to
$$\lim_{n\to\infty} \wt a_2 \wt a_3 \hdots \wt a_n Y_1 = 0 \as$$
On the other hand, let $T_n$ be defined as
$$T_n = \sum_{k=2}^{n} \wt a_{2} \wt a_{3} \ldots \wt a_{k-1} e_k$$
and $T$ given by
$$T= \sum_{k=2}^{\infty} \wt a_{2} \wt a_{3} \ldots \wt a_{k-1} e_k.$$ 
We have
\begin{align*}
\E[|T-T_n|] &= \E\left[\left|\sum_{k=n+1}^{\infty} \wt a_{2} \wt a_{3} \ldots \wt a_{k-1} e_k\right|\right],\\
	&\leq \sum_{k=n+1}^{\infty}\E\left[\left| \wt a_{2} \wt a_{3} \ldots \wt a_{k-1} e_k\right|\right],\\
	&\leq \E[|e_2|]\sum_{k=n+1}^{\infty}\E\left[\left| \wt a_{2}\right|\right]^{k-2}.
\end{align*}
In addition, $\E[a_n^2] < 1$ and $\E[b_n^2] < 1$ which leads to $\E[\wt a_n^2] < 1$ and $\E[|\wt a_n|] < 1$. Consequently,
$$\E[|T-T_n|] \leq \E\left[\left| \wt a_{2}\right|\right]^{n-1}\frac{\E[|e_2|]}{1-\E\left[\left| \wt a_{2}\right|\right]}.$$
This proves that $T_n \limnorm T$ which immediately implies that
$$T_n \liml T \hspace{20pt} \text{ and } \hspace{20pt} Y_n \liml T.$$
Moreover, we can easily see that \H{H1} allows us to say that $\E[T^{2}]<\infty$ thanks to the Cauchy-Schwarz inequality. It only remains to prove that $T$ is not degenerate. First, we easily have, since $\E[|\wt a_2|]<1$
\begin{align*}
\E[T] &= \E\left[\sum_{k=2}^{\infty} \wt a_{2} \wt a_{3} \ldots \wt a_{k-1} e_k\right] = \sum_{k=2}^{\infty}\E\left[ \wt a_{2} \wt a_{3} \ldots \wt a_{k-1} e_k\right],\\
&= \sum_{k=2}^{\infty}\E\left[ \wt a_{2}\right]\E\left[ \wt a_{3}\right] \ldots \E\left[\wt a_{k-1}\right]\E\left[ e_k\right] = \frac{c+d}{2-(a+b)}.
\end{align*}
Then, we can calculate $\E[T^2]$ as follows
\begin{align*}
\E[T^2] &= \E\left[\left(\sum_{k=2}^{\infty} \wt a_{2} \wt a_{3} \ldots \wt a_{k-1} e_k\right)^2\right],\\
&= \sum_{k=2}^\infty \E[\wt a_{2}^2 \wt a_{3}^2 \ldots \wt a_{k-1}^2 e_k^2] + 2\sum_{k=2}^\infty\sum_{l=k+1}^\infty\E[\wt a_{2}^2 \wt a_{3}^2 \ldots \wt a_{k-1}^2 \wt a_k e_k \wt a_{k+1} \ldots \wt a_{l-1}  e_l],\\
&= \sum_{k=2}^\infty \left(\frac{\sigma_a^2+\sigma_b^2+a^2+b^2}2\right)^{k-2}\frac{\sigma_c^2+\sigma_d^2+c^2+d^2}2\\
	&\hspace{30pt} +2\sum_{k=2}^\infty\sum_{l=k+1}^\infty \left(\frac{\sigma_a^2+\sigma_b^2+a^2+b^2}2\right)^{k-2}\frac{ac+bd}2\left(\frac{a+b}2\right)^{l-k-2}\frac{c+d}2,\\
&= \frac{\sigma_c^2+\sigma_d^2+c^2+d^2}{2-(\sigma_a^2+\sigma_b^2+a^2+b^2)} + \frac{2(ac+bd)(c+d)}{(2-(\sigma_a^2+\sigma_b^2+a^2+b^2))(2-(a+b))}.
\end{align*}
This allows us to say that
\begin{multline*}
\Var(T) = \frac{\sigma_c^2+\sigma_d^2}{2-\left(\sigma_a^2 + \sigma_b^2+a^2 +b^2\right)} + \left(\frac{c+d}{2-(a+b)}\right)^2\frac{\sigma_a^2+\sigma_b^2}{2-\left(\sigma_a^2 + \sigma_b^2+a^2 +b^2\right)}\\
	+\frac2{2-\left(\sigma_a^2 + \sigma_b^2+a^2 +b^2\right)} \frac{(ad-bc+c-d)^2}{(2-(a+b))^2}.
\end{multline*}
Under hypothesis \H{H1} and \H{H2} we immediately have that the first term is positive and that the two other terms are non-negative, allowing us to say that $T$ is not degenerate.


\section{Proof of Lemma \ref{LFGN}} \label{demoLFGN}


We shall now prove that for all $f\in\calC_{b}^1(\R_+)$,
\begin{equation*} \label{limLFGN}
\lim_{n\to\infty} \frac1{|\T_n|} \sum_{k\in\T_n} f(X_k) = \E[f(T)].
\end{equation*}
Denote $g=f-\E[f(T)]$,
$$\begin{array}{ccccc}
\displaystyle{\overline M_{\T_n} (f) = \frac1{|\T_n|} \sum_{k\in \T_n} f(X_k)} & & \text{and} & & \displaystyle{\overline M _{\G_n} (f) = \frac1{|\G_n|} \sum_{k\in \G_n} f(X_k)}.
\end{array}$$
Via Lemma A.2 of \cite{BercuBDSAGP}, it is only necessary to prove that
$$\lim_{n\to\infty} \frac1{|\G_n|} \sum_{k\in\G_n} g(X_k) = 0 \hspace{20pt} \text{a.s.}$$
We shall follow the induced Markov chain approach, originally proposed by Guyon in \cite{Guyon}. Let $Q$ be the transition probability of $(Y_n)$, $Q^p$ the $p$-th iterated of $Q$. In addition, denote by $\nu$ the distribution of $Y_1=X_1$ and $\nu Q^p$ the law of $Y_p$. Finally, let $P$ be the transition probability of $(X_n)$ as defined in \cite{Guyon}. We obtain from relation (7) of \cite{Guyon} that for all $n\geq0$
$$\E[\overline M _{\G_n} (g)^2] = \frac1{2^n} \nu Q^n g^2 + \sum_{k=0}^{n-1} \frac1{2^{k+1}} \nu Q^k P(Q^{n-k-1}g \star Q^{n-k-1}g)$$
where, for all $x,y\in\N$, $(f\star g)(x,y) = f(x)g(y)$. Consequently,
\begin{align}
\sum_{n=0}^\infty \E[\overline M _{\G_n} (g)^2] &= \sum_{n=0}^\infty  \frac1{2^n} \nu Q^n g^2 + \sum_{n=1}^\infty \sum_{k=0}^{n-1} \frac1{2^{k+1}} \nu Q^k P(Q^{n-k-1}g \star Q^{n-k-1}g),\nonumber\\
	&\leq \sum_{k=0}^\infty \frac1{2^k} \nu Q^k\left(g^2+P\left(\sum_{l=0}^\infty |Q^l g \star Q^l g|\right)\right). \label{etoile}
\end{align}
However, for all $x\in \N$,
$$Q^n g(x) = Q^n f(x) - \E[f(T)] = \E_x[f(Y_n) - f(T)] = \E_x[f(Z_n)-f(T)]$$
where $Z_n$ is given by \eqref{defZn}. Hence, we deduce from the mean value theorem and the Cauchy-Schwarz inequality that
\begin{equation}
|Q^n g(x)| \leq \E_x[W_n|Z_n-T|] \leq \E_x[W_n^2]^{1/2}\E_x[(Z_n-T)^2]^{1/2} \label{majQn}
\end{equation}
where
$W_n = \sup_{z\in[Z_n,T]}|f'(z)|$. By the very definition of $\calC_{b}^1(\R_+)$, one can find some constant $\gamma>0$ such that $|f'(z)|\leq \gamma$. Hence, 
\begin{equation}\label{espW}
\E_x[W_n^2]^{1/2} \leq \gamma.
\end{equation}
Furthermore 
$$\displaystyle Z_n- T= \wt a_2 \hdots \wt a_n Y_1 - \sum_{k=n}^\infty \wt a_2 \hdots \wt a_k e_{k+1}$$
and the triangle inequality allows us to say that 
\begin{align}
\E_x[(Z_n-T)^2]^{1/2} &\leq \E_x[(\wt a_2 \hdots \wt a_n Y_1)^2]^{1/2} + \sum_{k=n}^\infty \E_x[(\wt a_2 \hdots \wt a_k e_{k+1})^2]^{1/2}\nonumber\\
	&\leq \E[\wt a_2^2]^{(n-1)/2} \E_x[Y_1^2]^{1/2} + \sum_{k=n}^\infty \E_x[\wt a_2^2]^{(k-1)/2} \E[e_{k+1}^2]^{1/2}\nonumber\\
	&\leq \sqrt{\E[\wt a_2^2]}^{n-1}\(|x|+\frac{\E[e_2^2]^{1/2}}{1-\E[\wt a_2^2]^{1/2}}\)\nonumber\\
	&\leq \alpha \sqrt{\E[\wt a_2^2]}^{n} (1+ |x|) \label{majZ}
\end{align}
where 
$$\alpha = \max\(1,\frac{\E[e_2^2]^{1/2}}{1-\E[\wt a_2^2]^{1/2}}\).$$
\noindent Finally, we obtain from \eqref{majQn} together with \eqref{espW} and \eqref{majZ} that
$$ |Q^n g(x)| \leq \gamma \alpha \sqrt{\E[\wt a_2^2]}^{n-1} (1+ |x|).$$
Therefore,
\begin{equation} \label{doubleetoile}
P\left(\sum_{n=0}^\infty|Q^n g\star Q^n g|\right) \leq \frac{\gamma^2 \alpha^2}{1-\E[\wt a_2^2]}P(h\star h)
\end{equation}
where, for all $x\in\N$, $h(x)=1+|x|$. We are now in position to prove that
\begin{equation}\label{but7}
\E\left[\sum_{n=0}^\infty \overline M _{\G_n} (g)^2\right] < \infty.
\end{equation}
Let $G$ be be the random vector defined by $G(x) = (a_1x+\veps_2,b_1x+\veps_3)^t$. We can easily see from \H{H2} that it exists some constant $\beta>0$ such that
$$P(h\star h)(x) = \E[(h\star h)(G(x))] \leq \beta (1+x^2).$$
Consequently, since, for all $z\in\R$, $|g(z)|\leq2\gamma$, we obtain from \eqref{etoile} together with \eqref{doubleetoile} that
\begin{align}
\sum_{n=0}^\infty \E[\overline M_{\G_n}(g)^2] &\leq \sum_{k=0}^\infty \frac1{2^k}\(\E[g^2(Y_k)] + \frac{\beta\gamma^2 \alpha^2}{1-\E[\wt a_2^2]}(1+\E[Y_k^2]) \),\nonumber\\
	&\leq \(8\gamma^2 + \frac{\beta\gamma^2 \alpha^2}{1-\E[\wt a_2^2]}\)\(1+\sum_{k=0}^\infty \frac1{2^k} \E[Y_k^2]\).\label{sumE}
\end{align}
In addition, we also have
\begin{align}
\E[Y_k^2]^{1/2} &= \E[Z_k^2]^{1/2},\nonumber\\
	&\leq \E_x[(\wt a_2 \hdots \wt a_n Y_1)^2]^{1/2} + \sum_{k=2}^n \E_x[(\wt a_2 \hdots \wt a_{k-1} e_{k})^2]^{1/2},\nonumber\\
	&\leq \E[\wt a_2^2]^{(n-1)/2} \E_x[Y_1^2]^{1/2} + \sum_{k=2}^\infty \E_x[\wt a_2^2]^{(k-2)/2} \E[e_{k+1}^2]^{1/2},\nonumber\\
	&\leq \E[X_1^2]^{1/2} + \frac{\E[e_2^2]^{1/2}}{1-\E[\wt a_2^2]^{1/2}}.\label{espY}
\end{align}
Then, \eqref{sumE} and \eqref{espY} immediately lead to \eqref{but7}. Finally, the monotone convergence theorem implies that
$$\lim_{n\to\infty} \overline M _{\G_n} (g) = 0 \hspace{20pt} \as$$
which completes the proof of Lemma \ref{LFGN}.


\section{Proof of Proposition \ref{cvcrochet}}


The almost sure convergence \eqref{limcrochet} immediately follows from \eqref{defcrochetM} and \eqref{defLk} together with Lemma \ref{LFGN}. It only remains to prove that $\det(L)>0$ where the limiting matrix $L$ can be rewritten as
$L = \E\left[\Gamma \otimes \calC\right]$,
where
$$\begin{array}{ccccc}
 \Gamma = \begin{pmatrix} P(T) & Q(T) \\ Q(T) & R(T) \end{pmatrix}
 & & \text{ and } & &
 \calC = \displaystyle\frac1{(1+T^2)^2} \begin{pmatrix} T^2 & T \\ T & 1 \end{pmatrix}.
\end{array}$$

\noindent We have
\begin{align}
 L &= \E\left[\begin{pmatrix} \sigma_a^2 T^2 & \rho_{ab} T^2 \\ \rho_{ab} T^2 & \sigma_b^2 T^2 \end{pmatrix} \otimes \calC\right] + \E\left[\begin{pmatrix} \sigma_c^2 & \rho_{cd} \\ \rho_{cd} & \sigma_d^2 \end{pmatrix} \otimes \calC\right],\nonumber\\
	&= \begin{pmatrix} \sigma_a^2 & \rho_{ab} \\ \rho_{ab} & \sigma_b^2 \end{pmatrix} \otimes \E[T^2\calC] + \begin{pmatrix} \sigma_c^2 & \rho_{cd} \\ \rho_{cd} & \sigma_d^2 \end{pmatrix} \otimes \E[\calC]. \label{decompoL}
\end{align}

\noindent We shall prove that $\E[\calC]$ is a positive definite matrix and that $\E[T^2\calC]$ is a positive semidefinite matrix. Denote by $\lambda_1$ and $\lambda_2$ the two eigenvalues of the real symmetric matrix $\E[\calC]$. We clearly have
$$\lambda_1 + \lambda_2 = tr(\E[\calC]) = \E\left[\frac{T^2+1}{(1+T^2)^2}\right] > 0$$
and
$$\lambda_1 \lambda_2 = \det(\E[\calC]) = \E\left[\frac{T^2}{(1+T^2)^2}\right] \E\left[\frac{1}{(1+T^2)^2}\right] - \E\left[\frac{T}{(1+T^2)^2}\right]^2 \geq0$$
thanks to the Cauchy-Schwarz inequality and $\lambda_1 \lambda_2 = 0$ if and only if $T$ is degenerate, which is not the case thanks to 
Lemma \ref{CVYn}. Consequently, $\E[\calC]$ is a positive definite matrix. In the same way, we can prove that $\E[T^{2}\calC]$ is a positive semidefinite matrix. Since the Kronecker product of two positive semidefinite (respectively positive definite) matrices is a positive semidefinite (respectively positive definite) matrix, we deduce from \eqref{decompoL} that $L$ is positive definite as soon as $\rho_{cd}^2 < \sigma_c^2 \sigma_d^2$ and $\rho_{ab}^2 \leq \sigma_a^2 \sigma_b^2$ which is the case thanks to \H{H3}.


\section{Proof of Theorem \ref{CVpstheta}} \label{demoCVpstheta}


We will follow the same approach as in Bercu et al.~\cite{BercuBDSAGP}. For all $n\geq1$, let $\calV_n = M_n^t \Sigma_{n-1}^{-1} M_n = (\wh\theta_n - \theta)^t \Sigma_{n-1} (\wh \theta_n - \theta)$.
First of all, we have
\begin{align*}
 \calV_{n+1} &= M_{n+1}^t \Sigma_n^{-1} M_{n+1} = (M_n + \Delta M_{n+1})^t \Sigma_n^{-1} (M_n + \Delta M_{n+1}),\\
	      &= M_n^t \Sigma_n^{-1} M_n + 2 M_n^t \Sigma_n^{-1} \Delta M_{n+1} + \Delta M_{n+1}^t \Sigma_n^{-1} \Delta M_{n+1},\\
	      &= \calV_n - M_n^t(\Sigma_{n-1}^{-1} - \Sigma_n^{-1})M_n + 2 M_n^t \Sigma_n^{-1} \Delta M_{n+1} + \Delta M_{n+1}^t \Sigma_n^{-1} \Delta M_{n+1}.
\end{align*}

\noindent By summing over this identity, we obtain the main decomposition
\begin{equation}\label{egaliteVn}
 \calV_{n+1} + \calA_n = \calV_1 + \calB_{n+1} + \calW_{n+1}
\end{equation}
where
$$\calA_n = \sum_{k=1}^n M_k^t(\Sigma_{k-1}^{-1} - \Sigma_k^{-1})M_k,$$
$$\calB_{n+1} = 2 \sum_{k=1}^n M_k^t \Sigma_k^{-1} \Delta M_{k+1} \hspace{10pt} \text{ and } \hspace{10pt} \calW_{n+1} = \sum_{k=1}^n \Delta M_{k+1}^t \Sigma_k^{-1} \Delta M_{k+1}.$$

\begin{Lem} \label{lemCVVnAn}
 Assume that \H{H1} to \H{H3} are satisfied. Then, we have
\begin{equation} \label{CVWn}
 \lim_{n\to\infty} \frac{\calW_{n}}{n} = \frac12 tr((I_2 \otimes C)^{-1/2} L (I_2 \otimes C)^{-1/2}) \hspace{20pt} \text{ a.s.}
\end{equation}
where $C$ is the positive definite matrix given by \eqref{defA}. In addition, we also have
\begin{equation}\label{CVBn}
\calB_{n+1} = o(n) \hspace{20pt} \text{ a.s.}
\end{equation}
and
\begin{equation} \label{CVVnAn}
 \lim_{n\to\infty} \frac{\calV_{n+1} + \calA_n}{n} = \frac12 tr((I_2 \otimes C)^{-1/2} L (I_2 \otimes C)^{-1/2}) \hspace{20pt} \text{ a.s.}
\end{equation}
\end{Lem}

\begin{proof}

\noindent First of all, we have $\calW_{n+1} = \calT_{n+1} + \calR_{n+1}$ where
$$\calT_{n+1} = \sum_{k=1}^n \frac{\Delta M_{k+1}^t (I_2 \otimes C)^{-1} \Delta M_{k+1}}{|\T_k|},$$
$$\calR_{n+1} = \sum_{k=1}^n \frac{\Delta M_{k+1}^t (|\T_k| \Sigma_k^{-1} - (I_2 \otimes C)^{-1}) \Delta M_{k+1}}{|\T_k|}.$$

\noindent One can observe that $\calT_{n+1} = tr((I_2 \otimes C)^{-1/2} \calH_{n+1} (I_2 \otimes C)^{-1/2})$ where
$$ \calH_{n+1} = \sum_{k=1}^n \frac{\Delta M_{k+1} \Delta M_{k+1}^t}{|\T_k|}.$$

\noindent Our aim is to make use of the strong law of large numbers for martingale transforms, so we start by adding and subtracting a term involving 
the conditional expectation of $\Delta \calH _{n+1}$ given $\calF_n$. We have thanks to relation \eqref{defLk} that for all 
$n\geq0$, $\E[\Delta M_{n+1} \Delta M_{n+1}^t | \calF_n] = L_n$. Consequently, we can split $\calH_{n+1}$ into two terms
$$\calH_{n+1} = \sum_{k=1}^n \frac{L_k}{|\T_k|} + \calK_{n+1},\hspace{20pt} \text{where} \hspace{20pt} \calK_{n+1} = \sum_{k=1}^n \frac{\Delta M_{k+1} \Delta M_{k+1}^t - L_k}{|\T_k|}.$$
It clearly follows from convergence \eqref{limcrochet} that 
$$ \lim_{n\to\infty} \frac{L_n}{|\T_n|} = \frac12 L \hspace{20pt} \text{ a.s.}$$
Hence, Cesaro convergence immediately implies that
\begin{equation} \label{CesaroLk}
 \lim_{n\to\infty} \frac1n \sum_{k=1}^n \frac{L_k}{|\T_k|} = \frac12 L \hspace{20pt} \text{ a.s.}
\end{equation}
On the other hand, the sequence $(\calK_n)_{n\geq2}$ is obviously a square integrable martingale. Moreover, we have 
$$\Delta \calK_{n+1}= \calK_{n+1} - \calK_n = \frac1{|\T_n|} (\Delta M_{n+1} \Delta M_{n+1}^t - L_n).$$
For all $u\in\R^4$, denote $\calK_n(u) = u^t\calK_n u$. It follows from tedious but straightforward calculations, together with Lemma 
\ref{LFGN}, that the increasing process of the martingale $(\calK_n(u))_{n\geq2}$ satisfies $\langle \calK(u)\rangle_n =\calO(n)$ a.s. Therefore, we deduce from 
the strong law of large numbers for martingales that for all $u\in\R^4$, $\calK_n(u) = o(n)$ a.s.~leading to $\calK_n = o(n)$ a.s. Hence, 
we infer from \eqref{CesaroLk} that
\begin{equation} \label{CVHn}
 \lim_{n\to\infty} \frac{\calH_{n+1}}n = \frac12 L \hspace{20pt} \text{a.s.}
\end{equation}

\noindent Via the same arguments as in the proof of convergence \eqref{limcrochet}, we find that
\begin{equation}\label{CVSigman}
\lim_{n\to\infty} \frac{\Sigma_n}{|\T_n|} = I_2 \otimes C \hspace{20pt} \text{a.s.}
\end{equation}
where $C$ is the positive definite matrix given by \eqref{defA}. Then, we obtain from \eqref{CVHn} that
$$ \lim_{n\to\infty} \frac{\calT_n}n = \frac12 tr((I_2 \otimes C)^{-1/2} L (I_2 \otimes C)^{-1/2}) \hspace{20pt} \text{a.s.}$$
which allows us to say that $\calR_n = o(n)$ a.s.~leading to \eqref{CVWn}. We are now in position to prove \eqref{CVBn}. Let us recall that 
$$\calB_{n+1} = 2 \sum_{k=1}^n M_k^t \Sigma_k^{-1} \Delta M_{k+1} = 2 \sum_{k=1}^n M_k^t \Sigma_k^{-1} \Psi_k \xi_{k+1}.$$

\noindent Hence, $(\calB_n)_{n\geq2}$ is a square integrable martingale. In addition, we have
$$\Delta \calB_{n+1} = 2M_n^t\Sigma_n^{-1}\Delta M_{n+1}.$$

\noindent Thus
\begin{align*}
 \E[(\Delta \calB_{n+1})^2 | \calF_n] &= 4 \E[M_n^t\Sigma_n^{-1}\Delta M_{n+1} \Delta M_{n+1}^t\Sigma_n^{-1}M_n | \calF_n] \hspace{20pt} \text{a.s.} \\
	    &= 4 M_n^t\Sigma_n^{-1}\E[\Delta M_{n+1} \Delta M_{n+1}^t| \calF_n]\Sigma_n^{-1}M_n \hspace{20pt} \text{a.s.}\\
	    &= 4 M_n^t\Sigma_n^{-1} L_n \Sigma_n^{-1}M_n \hspace{20pt} \text{a.s.}
\end{align*}

\noindent We can observe that
\begin{equation*}
 L_n = \sum_{k\in\G_n} \frac1{c_k^2}\begin{pmatrix}
                        P(X_k) & Q(X_k) \\
			Q(X_k) & R(X_k)
                       \end{pmatrix}
			      \otimes \begin{pmatrix}
			               X_k^2 & X_k \\ X_k & 1
			              \end{pmatrix}
\end{equation*}
and
\begin{equation*}
\Psi_n\Psi_n^t = \sum_{k\in\G_n} \frac{1}{c_k} I_2 \otimes \begin{pmatrix}
						X_k^2 & X_k \\ X_k & 1
					      \end{pmatrix}.
\end{equation*}
For $\alpha=\max(\sigma_a^2,\sigma_c^2)+\max(\sigma_b^2,\sigma_d^2)+\max(|\rho_{ab}|,|\rho_{cd}|)$, denote
$$\Delta_n=\begin{pmatrix} \alpha - \displaystyle\frac{P(X_n)}{c_n} & - \displaystyle\frac{Q(X_n)}{c_n} \vspace{5pt}\\ - \displaystyle\frac{Q(X_n)}{c_n} & \alpha - \displaystyle\frac{R(X_n)}{c_n} \end{pmatrix}.$$
We can rewrite $\alpha\Psi_n \Psi_n^t - L_n$ as
$$\alpha\Psi_n \Psi_n^t - L_n= \sum_{k\in \G_n} \frac1{c_k} \Delta_k \otimes \begin{pmatrix} X_k^2 & X_k \\ X_k & 1 \end{pmatrix}.$$
It is not hard to see that $\Delta_n$ is a positive definite matrix. As a matter of fact, we deduce from the elementary inequalities
\begin{equation}\label{majpoly}
\begin{cases}
0<P(X)\leq \max(\sigma_a^2,\sigma_c^2)(1+X^2),\\
0<R(X)\leq \max(\sigma_b^2,\sigma_d^2)(1+X^2),\\
|Q(X)| \leq \max(|\rho_{ab}|,|\rho_{cd}|)(1+X^2),
\end{cases}
\end{equation}
that
$$tr(\Delta_n) = 2\alpha - \frac{P(X_n)}{c_n} -\frac{R(X_n)}{c_n} \geq 2\alpha - \max(\sigma_a^2,\sigma_c^2)-\max(\sigma_b^2,\sigma_d^2) >0.$$
In addition, we also have from \eqref{majpoly} that
\begin{align*}
c_n^2\det(\Delta_n) &= (\alpha c_n - P(X_n))(\alpha c_n - R(X_n)) - Q^2(X_n),\\
	&=\alpha c_n\left(\alpha c_n - P(X_n) - R(X_n)\right) + P(X_n)R(X_n) - Q^2(X_n),\\
	&\geq P(X_k)R(X_k) + \alpha c_n^2 \max(|\rho_{ab}|,|\rho_{cd}|) - Q^2(X_n),\\
	&\geq P(X_k)R(X_k) + \max(|\rho_{ab}|,|\rho_{cd}|)^2 c_n^2 - Q^2(X_n) > 0.
\end{align*}
Consequently, $\Delta_n$ is positive definite which immediately implies that $L_n \leq \alpha \Psi_n\Psi_n^t$. Moreover, we can use Lemma B.1 of \cite{BercuBDSAGP} to say that 
$$\Sigma_{n}^{-1} \Psi_n \Psi_n^t \Sigma_n^{-1} \leq \Sigma_{n-1}^{-1} - \Sigma_n^{-1}.$$

\noindent Hence
\begin{align*}
 \E[(\Delta \calB_{n+1})^2 | \calF_n] &= 4 M_n^t\Sigma_n^{-1} L_n \Sigma_n^{-1}M_n \hspace{20pt} \text{a.s.}\\
			&\leq 4 \alpha M_n^t\Sigma_n^{-1} \Psi_n\Psi_n^t \Sigma_n^{-1}M_n \hspace{20pt} \text{a.s.}\\
			&\leq 4 \alpha M_n^t(\Sigma_{n-1}^{-1} - \Sigma_n^{-1})M_n \hspace{20pt} \text{a.s.}\\
\end{align*}

\noindent leading to $\langle \calB\rangle_n \leq 4\alpha \calA_n$. Therefore it follows from the strong law of large numbers for martingales that $\calB_n = o(\calA_n)$. Hence, we deduce from decomposition \eqref{egaliteVn} that 
$$\calV_{n+1} + \calA_n = o(\calA_n) + \calO(n) \hspace{20pt} \text{ a.s.}$$

\noindent leading to $\calV_{n+1} = \calO(n)$ and $\calA_n = \calO(n)$ a.s.~which implies that $\calB_n= o(n)$ a.s. Finally we clearly obtain convergence \eqref{CVVnAn} from the main decomposition \eqref{egaliteVn} together with \eqref{CVWn} and \ref{CVBn}, which completes the proof of Lemma \ref{lemCVVnAn}.
\end{proof}

\begin{Lem}\label{delta}
 Assume that \H{H1} to 
\H{H5} are satisfied. For all $\delta > 1/2$, we have
\begin{equation*}\label{majdelta}\|M_n\|^2 = o(|\T_n|n^\delta) \hspace{20pt} \text{ a.s.}\end{equation*}
\end{Lem}

\begin{proof}
Let us recall that
$$M_n = \sum_{k\in\T_{n-1}} \frac1{c_k} \begin{pmatrix} X_kV_{2k} \\ V_{2k} \\ X_kV_{2k+1} \\V_{2k+1}\end{pmatrix}.$$
Denote
$$\begin{array}{ccccc}P_n = \displaystyle \sum_{k\in\T_{n-1}} \frac{X_kV_{2k}}{c_k}  & & \text{ and } & & \displaystyle Q_n = \sum_{i\in\T_{n-1}} \frac{V_{2k}}{c_k}.\end{array}$$
On the one hand, $P_n$ can be rewritten as
$$\begin{array}{ccccc}\displaystyle{ P_n = \sum_{k=1}^n \sqrt{|\G_{k-1}|} f_k } & & \text{ where } & & \displaystyle f_n = \frac1{\sqrt{|\G_{n-1}|}}\sum_{k\in\G_{n-1}} \frac{X_kV_{2k}}{c_k}.\end{array}$$
We already saw in Section 3 that for all $k\in\G_n$,
$$\begin{array}{ccccccc} \E[V_{2k}|\calF_n] = 0 & & \text{ and } & & \E[V_{2k}^2|\calF_n] = \sigma_a^2X_k^2+\sigma_c^2 = P(X_k) & & \text{a.s.} \end{array}$$
In addition, for all $k\in\G_n$,
$$\E[V_{2k}^4|\calF_n] = \mu_a^4X_k^4 + 6\sigma_a^2\sigma_c^2 X_k^2 + \mu_c^4 \hspace{20pt} \text{a.s.}$$
which implies that
\begin{equation}\label{majV2k4} \E[V_{2k}^4|\calF_n] \leq \mu_{ac}^4 c_k^2 \hspace{20pt} \text{a.s.}.\end{equation}
where $\mu_{ac}^4 = \max(\mu_a^4,3\sigma_a^2\sigma_c^2,\mu_c^4)$. Consequently, $\E[f_{n+1}|\calF_n]=0$ a.s.~and we deduce from \eqref{majV2k4} together with the Cauchy-Schwarz inequality that
\begin{align}
\E[f_{n+1}^4|\calF_n] &= \frac1{|\G_n|} \E\left[\left. \(\sum_{k\in\G_n} \frac{X_kV_{2k}}{c_k} \)^4 \right|\calF_n\right],\nonumber\\
	&= \frac1{|\G_n|^2} \sum_{k\in\G_n}\(\frac{X_k}{\sqrt{c_k}}\)^4\frac{\E[V_{2k}^4|\calF_n]}{c_k^2}\nonumber\\
		&\hspace{80pt}+ \frac3{|\G_n|^2} \sum_{k\in\G_n}\sum_{\substack{l\in\G_n \\ l\neq k}}\(\frac{X_k}{\sqrt{c_k}}\)^2\(\frac{X_l}{\sqrt{c_l}}\)^2\frac{\E[V_{2k}^2|\calF_n]}{c_k} \frac{\E[V_{2l}^2|\calF_n]}{c_l}, \nonumber\\
	&\leq \frac1{|\G_n|^2} \sum_{k\in\G_n} \mu_{ac}^4 + \frac3{|\G_n|^2} \sum_{k\in\G_n}\sum_{\substack{l\in\G_n \\ l\neq k}} \max(\sigma_a^2,\sigma_c^2)^2,\nonumber\\
&\leq \mu_{ac}^4 + 3 \max(\sigma_a^2,\sigma_c^2)^2 \hspace{20pt} \text{a.s.} \label{majfn}
\end{align}
Therefore, we infer from \eqref{majfn} that $\sup_{n\geq0}\E[f_{n+1}^4|\calF_n]<\infty$ a.s.
Hence, we obtain from Wei's Lemma given in \cite{Wei} page 1672 that for all $\delta > 1/2$,
$$P_n^2 = o(|\T_{n-1}|n^\delta) \hspace{20pt} \text{a.s.}$$
On the other hand, $Q_n$ can be rewritten as
$$\begin{array}{ccccc}\displaystyle{ Q_n = \sum_{k=1}^n \sqrt{|\G_{k-1}|} g_k } & & \text{ where } & & \displaystyle g_n = \frac1{\sqrt{|\G_{n-1}|}}\sum_{k\in\G_{n-1}} \frac{V_{2k}}{c_k}.\end{array}$$
Via the same calculation as before, $\E[g_{n+1}|\calF_n] = 0$ a.s.~and, as $c_n\geq1$,
$$\E[g_{n+1}^4|\calF_n]\leq \mu_{bd}^4 + 3 \max(\sigma_b^2,\sigma_d^2)^2 \hspace{20pt} \text{a.s.}$$
Hence, we deduce once again from Wei's Lemma that for all $\delta > 1/2$,
$$Q_n^2 = o(|\T_{n-1}|n^\delta) \hspace{20pt} \text{a.s.}$$
In the same way, we obtain the same result for the two last components of $M_n$, which completes the proof of Lemma \ref{delta}.
\end{proof}

\noindent {\bf Proof of Theorem \ref{CVpstheta}.} We recall from \eqref{difftheta} that $\wh \theta_n-\theta = \Sigma_{n-1}^{-1} M_n$ which implies
$$\|\wh \theta_n - \theta \|^2 \leq \frac{\calV_n}{\lambda_{min}(\Sigma_{n-1})}$$
where $\calV_n = M_n^t \Sigma_{n-1}^{-1} M_n$. On the one hand, it follows from \eqref{CVVnAn} that $\calV_n=\calO(n)$ a.s. On the other hand, we deduce from \eqref{CVSigman} that
$$\lim_{n\to\infty} \frac{\lambda_{min}(\Sigma_n)}{|\T_n|} = \lambda_{min}(C) >0 \hspace{20pt} \text{a.s.}$$
Consequently, we find that
$$\|\wh \theta_n - \theta \|^2 = \calO\left(\frac{n}{|\T_{n-1}|}\right) \hspace{20pt} \text{a.s.}$$

\noindent We are now in position to prove the quadratic strong law \eqref{quadratic1}. First of all a direct application of Lemma \ref{delta} ensures that $\calV_n = o(n^\delta)$ a.s.~for all $\delta > 1/2$. Hence, we obtain from \eqref{CVVnAn} that
\begin{equation}\lim_{n\to\infty} \frac{\calA_n}n = \frac12 tr((I_2 \otimes C)^{-1/2} L (I_2 \otimes C)^{-1/2}) \hspace{20pt} \text{ a.s.}\label{CVAn/n}\end{equation}

\noindent Let us rewrite $\calA_n$ as 
$$\calA_n = \sum_{k=1}^n M_k^t\left(\Sigma_{k-1}^{-1} - \Sigma_k^{-1}\right) M_k = \sum_{k=1}^n M_k^t\Sigma_{k-1}^{-1/2} A_k \Sigma_{k-1}^{-1/2} M_k$$

\noindent where $A_k = I_4 - \Sigma_{k-1}^{1/2}\Sigma_{k}^{-1}\Sigma_{k-1}^{1/2}$. We already saw from \eqref{CVSigman} that 
$$\lim_{n\to\infty} \frac{\Sigma_n}{|\T_n|} = I_2\otimes C \hspace{20pt} \text{a.s.}$$
which ensures that
$$\displaystyle \lim_{n\to\infty} A_n = \frac12 I_4 \hspace{20pt} \text{a.s.}$$
In addition, we deduce from \eqref{CVVnAn} that $\calA_n=\calO(n)$ a.s.~which implies that
\begin{equation}\frac{\calA_n}n = \left(\frac1{2n}\sum_{k=1}^nM_k^t\Sigma_{k-1}^{-1}M_k\right) + o(1) \hspace{20pt} \text{ a.s.}\label{eqAn/n}\end{equation}

\noindent Moreover we have
\begin{align}
 \frac1n\sum_{k=1}^nM_k^t\Sigma_{k-1}^{-1}M_k &= \frac1n \sum_{k=1}^n (\wh \theta_k - \theta)^t\Sigma_{k-1}(\wh \theta_k - \theta), \nonumber\\
	  &= \frac1n \sum_{k=1}^n |\T_{k-1}| (\wh \theta_k - \theta)^t\frac{\Sigma_{k-1}}{|\T_{k-1}|}(\wh \theta_k - \theta), \nonumber\\
	  &= \frac1n \sum_{k=1}^n |\T_{k-1}| (\wh \theta_k - \theta)^t(I_2 \otimes C)(\wh \theta_k - \theta) + o(1) \hspace{20pt} \text{ a.s.} \label{elementQSL}
\end{align}

\noindent Therefore, \eqref{CVAn/n} together with \eqref{eqAn/n} and \eqref{elementQSL} lead to \eqref{quadratic1}.


\section{Proof of Theorem \ref{CVpsvar}} \label{demoCVpsvar}

First of all, we shall only prove \eqref{vitesseeta} since the proof of \eqref{vitesseetad} follows exactly the same lines. We clearly have from \eqref{esteta} that
\begin{align} 
Q_{n-1}(\wh \eta_{n} - \eta_{n}) &= \sum_{k\in\T_{n-1}} \frac1{d_k}(\wh V_{2k}^2 - V_{2k}^2) \psi_k,\nonumber\\
	&=\sum_{l=0}^{n-1}\sum_{k\in\G_{l}} \frac1{d_k}(\wh V_{2k}^2 - V_{2k}^2) \psi_k,\nonumber\\
	&= \sum_{l=0}^{n-1} \sum_{k\in\G_l} \frac1{d_k}\left((\wh V_{2k} - V_{2k})^2 + 2(\wh V_{2k} - V_{2k})V_{2k}\right) \psi_k.\label{diffeta}
\end{align}
In addition, we already saw in Section 3 that for all $l\geq0$ and $k\in\G_l$,
$$\wh V_{2k} - V_{2k} = -\begin{pmatrix} \wh a_l-a \\ \wh c_l - c \end{pmatrix}^t \Phi_k.$$
Consequently,
$$(\wh V_{2k} - V_{2k})^2 \leq \|\Phi_k\|^2 \left((\wh a_l-a )^2 + (\wh c_l-c )^2\right)=c_k \left((\wh a_l-a )^2 + (\wh c_l-c )^2\right).$$
Hence, as $\|\psi_k\|^2 = X_k^4+1 \leq c_k^2$,
\begin{align*}
\left\|\sum_{l=0}^{n-1} \sum_{k\in\G_l} \frac{(\wh V_{2k} - V_{2k})^2}{d_k} \psi_k\right\| &\leq \sum_{l=0}^{n-1} \sum_{k\in\G_l} \frac{c_k\|\psi_k\|}{d_k} \left((\wh a_l-a )^2 + (\wh c_l-c )^2\right),\nonumber\\
	&\leq \sum_{l=0}^{n-1} |\G_l|\left((\wh a_l-a )^2 + (\wh c_l-c )^2\right) .\label{majterme1}
\end{align*}
However, as $\Lambda$ is positive definite, we obtain from \eqref{quadratic1} that 
$$\sum_{l=0}^{n-1} |\G_{l}| \left((\wh a_l-a )^2 + (\wh c_l-c )^2\right) = \calO(n) \hspace{20pt} \text{a.s.}$$
which implies that 
\begin{equation}
\left\|\sum_{l=0}^{n-1} \sum_{k\in\G_l} \frac{(\wh V_{2k} - V_{2k})^2}{d_k} \psi_k\right\| = \calO(n) \hspace{20pt} \text{a.s.} \label{Oterme1}
\end{equation}
Furthermore, denote
$$P_n = \sum_{l=0}^{n-1} \sum_{k\in\G_l} \frac{(\wh V_{2k} - V_{2k})V_{2k}}{d_k} \psi_k.$$
We clearly have
\begin{align*}
\Delta P_{n+1} = P_{n+1} - P_n = \sum_{k\in\G_n} \frac{(\wh V_{2k} - V_{2k})V_{2k}}{d_k} \psi_k,
	=-\sum_{k\in\G_n} \frac{V_{2k}}{d_k} \psi_k \Phi_k^t \begin{pmatrix} \wh a_{n}-a \\ \wh c_{n} - c \end{pmatrix}.
\end{align*}
In addition, for all $k\in\G_n$, $\E[V_{2k}|\calF_n] = 0$ a.s.~and $\E[V_{2k}^2|\calF_n] = \sigma_a^2 X_k^{2} +\sigma_c^2 \leq \alpha c_k$ a.s.~where $\alpha=\max(\sigma_a^2,\sigma_c^2)$. Consequently, $\E[\Delta P_{n+1} |\calF_n] = 0$ a.s.~and
\begin{align*}
\E[\Delta P_{n+1} \Delta P_{n+1}^t |\calF_n] &= \sum_{k\in\G_n} \frac1{d_k^2} \E[V_{2k}^2|\calF_n] \psi_k \Phi_k^t \begin{pmatrix} \wh a_{n} - a \\ \wh c_{n} - c \end{pmatrix}\begin{pmatrix} \wh a_{n} - a \\ \wh c_{n} - c \end{pmatrix}^t \Phi_k \psi_k^t \as\\
	&=\sum_{k\in\G_n} \frac{P(X_k)}{d_k^2} \psi_k \Phi_k^t \begin{pmatrix} \wh a_{n} - a \\ \wh c_{n} - c \end{pmatrix}\begin{pmatrix} \wh a_{n} - a \\ \wh c_{n} - c \end{pmatrix}^t \Phi_k \psi_k^t \as
\end{align*}
Therefore, $(P_n)$ is a square integrable vector martingale with increasing process $\langle P\rangle_n$ given by
\begin{align*}
\langle P\rangle _n &= \sum_{l=1}^{n-1} \E[\Delta P_{l+1} \Delta P_{l+1}^t |\calF_l] \hspace{20pt} \text{a.s.}\\
	&= \sum_{l=1}^{n-1} \sum_{k\in\G_l}\frac{P(X_k)}{d_k^2} \psi_k \Phi_k^t \begin{pmatrix} \wh a_l - a \\ \wh c_l - c \end{pmatrix}\begin{pmatrix} \wh a_l - a \\ \wh c_l - c \end{pmatrix}^t \Phi_k \psi_k^t \as
\end{align*}
It immediately follows from the previous calculation that
\begin{align*}
\|\langle P\rangle _n\| &\leq \alpha \sum_{l=0}^{n-1} \left((\wh a_l-a )^2 + (\wh c_l-c )^2\right)\sum_{k\in\G_l}\frac{c_k \|\psi_k\|^2 \|\Phi_k\|^2}{d_k^2} \hspace{20pt} \text{a.s.}\\
&\leq \alpha \sum_{l=0}^{n-1} |\G_l|\left((\wh a_l-a )^2 + (\wh c_l-c )^2\right) \hspace{20pt} \text{a.s.}
\end{align*}
leading to $\|\langle P\rangle _n\| = \calO(n)$ a.s. Then, we deduce from the strong law of large numbers for martingale given e.g.~in Theorem 1.3.15 of \cite{Duflo} that
\begin{equation}
P_n = o(n) \hspace{20pt} \text{a.s.} \label{oPn}
\end{equation}
Hence, we find from \eqref{diffeta}, \eqref{Oterme1} and \eqref{oPn} that
$\|Q_{n-1}(\wh \eta_{n} - \eta_{n})\| = \calO(n)$ a.s.
Moreover, we infer once again from Lemma \ref{LFGN} that
\begin{equation}
\lim_{n\to\infty} \frac1{|\T_n|}Q_n = D = \E\left[\frac1{(1+T^2)^2}\begin{pmatrix} T^4 & T^2 \\ T^2 & 1 \end{pmatrix}\right] \hspace{20pt} \text{ a.s.} \label{CVQ}
\end{equation}
Moreover, we can prove through tedious calculations that $T^2$ is not degenerate which allows us to say that $D$ is positive definite. This ensures that
$$\|\wh \eta_{n} - \eta_{n}\| = \calO\left(\frac n{|\T_{n-1}|}\right) \hspace{20pt} \text{a.s.}$$

\noindent It remains to establish \eqref{vitesserho}. Denote
$$\begin{array}{ccccc}
\wh W_n = \begin{pmatrix} \wh V_{2n}\\ \wh V_{2n+1} \end{pmatrix} & & \text{ and } & & R_n = \displaystyle\sum_{k\in\T_{n-1}} \frac1{d_k}\left( \wh W_k - W_k \right)^t J W_k \psi_k
\end{array}$$
where $J = \begin{pmatrix} 0 & 1 \\ 1 & 0 \end{pmatrix}$.
Then, we have from \eqref{carre} that
$$Q_{n-1} (\wh \nu_n - \nu_n) = \sum_{k\in\T_{n-1}} \frac1{d_k}\left(\wh V_{2k} - V_{2k} \right)\left(\wh V_{2k+1} - V_{2k+1} \right)\psi_k  + R_n.$$
 It is not hard to see that $(R_n)$ is a square integrable real martingale with increasing process given by
\begin{align*}
\langle R\rangle _n &= \sum_{l=0}^{n-1} \sum_{k\in\G_l} \E\left[\frac1{d_k^2}\left. (\wh W_k - W_k)^t J W_kW_k^t J (\wh W_k - W_k) \psi_k \psi_k^t\right|\calF_{l} \right] \hspace{20pt} \text{a.s.}\\
	  &= \sum_{l=0}^{n-1} \sum_{k\in\G_l}\frac1{d_k^2} (\wh W_k - W_k)^t J \E\left[\left. W_kW_k^t\right|\calF_l \right] J (\wh W_k - W_k)\psi_k \psi_k^t \hspace{20pt} \text{a.s.}\\
	  &= \sum_{l=0}^{n-1} \sum_{k\in\G_l}\frac1{d_k^2} (\wh W_k - W_k)^t J \begin{pmatrix} P(X_k) & Q(X_k) \\ Q(X_k) & R(X_k) \end{pmatrix} J (\wh W_k - W_k)\psi_k \psi_k^t \hspace{20pt} \text{a.s.}\\
	  &= \sum_{l=0}^{n-1} \sum_{k\in\G_l}\frac1{d_k^2} (\wh W_k - W_k)^t \begin{pmatrix} R(X_k) & Q(X_k) \\ Q(X_k) & P(X_k) \end{pmatrix} (\wh W_k - W_k)\psi_k \psi_k^t \hspace{20pt} \text{a.s.}
\end{align*}
Consequently, Lemma \ref{LFGN} together with \eqref{quadratic1} allows us to say that $\|\langle R\rangle _n\|=\calO(n)$ a.s.~which ensures that $R_n=o(n)$ a.s. Moreover,
\begin{align*}
&\left\|\sum_{k\in\T_{n-1}}\frac1{d_k} \left(\wh V_{2k} - V_{2k} \right)\left(\wh V_{2k+1} - V_{2k+1} \right)\psi_k\right\|\\
&\hspace{120pt} \leq \frac12 \sum_{k\in\T_{n-1}} \frac1{d_k}\left(\left(\wh V_{2k} - V_{2k} \right)^2+\left(\wh V_{2k+1} - V_{2k+1} \right)^2\right)\|\psi_k\|,\\
	&\hspace{120pt}\leq \frac12 \sum_{l=0}^{n-1} \|\wh \theta_l - \theta\|^2 \sum_{k\in\G_l} \frac{\|\Phi_k\|^2\|\psi_k\|}{d_k},\\
	&\hspace{120pt}\leq \frac12 \sum_{l=0}^{n-1} |\G_l| \|\wh \theta_l - \theta\|^2,
\end{align*}
which implies via \eqref{quadratic1} that
$$\left\|\sum_{k\in\T_{n-1}} \left(\wh V_{2k} - V_{2k} \right)\left(\wh V_{2k+1} - V_{2k+1} \right) \psi_k \right\| = \calO(n) \hspace{20pt} \text{a.s.}$$
Therefore, we obtain that $ \|Q_{n-1}(\wh \nu_n - \nu_n)\| = \calO(n)$ a.s.
which leads to \eqref{vitesserho}. Finally, it only remains to prove the a.s.~convergence of $\eta_n$, $\zeta_n$ and $\nu_n$ to $\eta$, $\zeta$ and $\nu$ which will immediately lead to the a.s.~convergence of $\wh \eta_n$, $\wh \zeta_n$ and $\wh \nu_n$ through \eqref{vitesseeta}, \eqref{vitesseetad} and \eqref{vitesserho}, respectively. On the one hand,
\begin{equation}
Q_{n-1} (\eta_n-\eta) = N_n = \sum_{k\in\T_{n-1}} \frac1{d_k} v_{2k} \psi_k \label{diffetan}
\end{equation}
where we recall that $v_{2n} = V_{2n}^2 - \eta^t\psi_n$. It is clear that $(N_n)$ is a square integrable vector martingale with increasing process $\langle N\rangle _n$ given by
\begin{align*}\langle N\rangle_n &= \sum_{l=0}^{n-1} \sum_{k\in\G_l} \frac1{d_k^2} \E[v_{2k}^2|\calF_l] \psi_k \psi_k^t \hspace{20pt} \text{a.s.}\\
	&\leq \gamma \sum_{l=0}^{n-1} \sum_{k\in\G_l} \frac1{d_k} \psi_k \psi_k^t \hspace{20pt} \text{a.s.}
\end{align*}
where $\gamma = \max(\mu_a^4-\sigma_a^4,2\sigma_a^2\sigma_c^2,\mu_c^4-\sigma_c^4)$. Hence,
$$\|\langle N\rangle _n\| \leq \gamma \sum_{l=0}^{n-1} \sum_{k\in\G_l} \frac1{d_k} \|\psi_k\|^2 \leq \gamma |\T_{n-1}| \hspace{20pt} \text{a.s.}$$
which immediately leads to
$\|\langle N\rangle _n\| = \calO(|\T_{n-1}|)$ a.s. Consequently,
$\|N_n\|^2 = \calO(n|\T_{n-1}|)$ a.s.~which leads via \eqref{CVQ} and \eqref{diffetan} to the a.s.~convergence of $\eta_n$ to $\eta$ and to the rate of convergence of Remark \ref{remrate}. The proof of the a.s.~convergence of $\zeta_n$ to $\zeta$ follows exactly the same lines. On the other hand
\begin{equation}
Q_{n-1}(\nu_n-\nu) = H_n = \sum_{k\in\T_{n-1}}\frac1{d_k}w_{2k}\psi_k \label{diffrhon}
\end{equation}
where we recall that $w_{2k} = V_{2k}V_{2k+1} - \E[V_{2k}V_{2k+1} | \calF_n]$. It is obvious to see that $(H_n)$ is a square integrable real martingale with increasing process 
\begin{align*}
\langle H\rangle _n &= \sum_{l=0}^{n-1} \sum_{k\in\G_l} \frac1{d_k^2} \E[w_{2k}^2 |\calF_l] \psi_k \psi_k^t \as\\
	&\leq \alpha \sum_{l=0}^{n-1} \sum_{k\in\G_l} \frac1{d_k}\psi_k\psi_k^t\as
\end{align*}
where $\alpha = \max\(\nu_{ab}^2, \nu_{cd}^2,\(\sigma_a^2+\sigma_c^2\)\(\sigma_b^2+\sigma_d^2\)\)$. This implies that
$$\|\langle H\rangle _n\| \leq \alpha \sum_{l=0}^{n-1} \sum_{k\in\G_l} \frac1{d_k} \|\psi_k\|^2 \leq \alpha |\T_{n-1}|\as$$
which allows us to say that
$$\|H_n\|^2 = \calO(n|\T_{n-1}|) \hspace{20pt} \text{and} \hspace{20pt} \|\wh\nu_n-\nu\|^2 = \calO\(\frac{n}{|\T_{n-1}|}\) \as$$
Finally, we deduce from \eqref{diffrhon} that $\nu_n$ converges a.s.~to $\nu$ and that the rate of convergence of Remark \ref{remrate} is verified, which completes the proof of Theorem \ref{CVpsvar}.


\section{Proof of Theorem \ref{TCL}}\label{demoTCL}


In order to establish the asymptotic normality of our estimators, we will extensively make use of the central limit theorem for triangular arrays of vector martingales given e.g.~by Theorem 2.1.9 of \cite{Duflo}. First of all, instead of using the generation-wise filtration $(\calF_n)$, we will use the sister pair-wise filtration $(\calG_n)$ given by 
$$\calG_n = \sigma(X_1,(X_{2k},X_{2k+1}),1\leq k \leq n ).$$

\noindent {\bf Proof of Theorem \ref{TCL}, first part.} We focus our attention to the proof of the asymptotic normality \eqref{TCLtheta}. Let $M^{(n)} = (M_k^{(n)})$ be the square integrable vector martingale defined as
\begin{equation}
M_k^{(n)} = \frac1{\sqrt{|\T_n|}}\sum_{i=1}^k D_i \label{defnewmart}, \hspace{20pt} \text{where} \hspace{20pt} D_i = \frac1{c_i}\begin{pmatrix}X_iV_{2i}\\V_{2i}\\X_iV_{2i+1}\\V_{2i+1}\end{pmatrix}.
\end{equation}
We clearly have
\begin{equation}
M_{t_n}^{(n)} = \frac1{\sqrt{|\T_n|}}\sum_{i=1}^{t_n} D_i = \frac1{\sqrt{|\T_n|}} M_{n+1}, \label{liennewmart}
\end{equation}
where $t_n=|\T_n|$. Moreover, the increasing process associated to $(M_k^{(n)})$ is given by
\begin{align*}
\langle M^{(n)}\rangle _{k} &= \frac1{|\T_n|} \sum_{i=1}^k \E\left[D_iD_i^t|\calG_{i-1}\right],\\
&= \frac1{|\T_n|} \sum_{i=1}^k \frac1{c_i^2} \begin{pmatrix}
                                P(X_i) & Q(X_i) \\ Q(X_i) & R(X_i) \end{pmatrix} \otimes \begin{pmatrix}X_i^2 & X_i\\X_i & 1 \end{pmatrix} \hspace{20pt} \text{a.s.}
\end{align*}
Consequently, it follows from convergence \eqref{limcrochet} that
$$\lim_{n\to\infty} \langle M^{(n)}\rangle _{t_n} = L \hspace{20pt} \text{a.s.}$$
It is now necessary to verify Lindeberg's condition by use of Lyapunov's condition. Denote
$$\phi_n = \sum_{k=1}^{t_n} \E\left[\left.\|M_k^{(n)} - M_{k-1}^{(n)} \|^4 \right| \calG_{k-1}\right].$$
We obtain from \eqref{defnewmart} that
\begin{align*}
\phi_n &= \frac1{|\T_n|^2} \sum_{k=1}^{t_n} \E\left[\left.\frac{(1+X_k^2)^2}{c_k^4}(V_{2k}^2 + V_{2k+1}^2)^2\right|\calG_{k-1}\right],\\
	&\leq \frac2{|\T_n|^2} \sum_{k=1}^{t_n} \frac1{c_k^2}\left(\E[V_{2k}^4|\calG_{k-1}] + \E[V_{2k+1}^4|\calG_{k-1}]\right).
\end{align*}
In addition, we already saw in Section \ref{demoCVpstheta} that
$$\E[V_{2n}^4|\calG_{n-1}] \leq \mu_{ac}^4 c_n^2, ~~~~~ \E[V_{2n+1}^4|\calG_{n-1}] \leq \mu_{bd}^4 c_n^2 \hspace{20pt} \text{a.s.}$$
where $\mu_{ac}^4 = \max(\mu_a^4,3\sigma_a^2\sigma_c^2,\mu_c^4)$ and $\mu_{bd}^4 = \max(\mu_b^4,3\sigma_b^2\sigma_d^2,\mu_d^4)$. Hence,
$$\phi_n \leq \frac{2(\mu_{ac}^4+\mu_{bd}^4)}{|\T_n|} \hspace{20pt} \text{a.s.}$$
which immediately implies that
$$\lim_{n\to\infty} \phi_n = 0 \hspace{20pt} \text{a.s.}$$
Therefore, Lyapunov's condition is satisfied and Theorem 2.1.9 of \cite{Duflo} allows us to say via \eqref{liennewmart} that
$$\frac1{\sqrt{|\T_{n-1}|}} M_n \liml \calN(0,L).$$
Finally, we infer from \eqref{difftheta} together with \eqref{CVSigman} and Slutsky's lemma that
$$\phantom{\square} \hspace{108pt} \sqrt{|\T_{n-1}|}(\wh\theta_n - \theta) \liml \calN(0,\Lambda^{-1}L\Lambda^{-1}).\hspace{108pt} \square$$

{\bf Proof of Theorem \ref{TCL}, second part.} We shall now establish the asymptotic normality given by \eqref{TCLeta}. Denote by $N^{(n)} = (N_k^{(n)})$ the square integrable vector martingale defined as
\begin{equation*}
N_k^{(n)} = \frac1{\sqrt{|\T_n|}}\sum_{i=1}^k \frac{v_{2i}}{d_i} \psi_i.
\end{equation*}
We immediately see from \eqref{diffetan} that
\begin{equation}N_{t_n}^{(n)} = \frac1{\sqrt{|\T_n|}} Q_n(\eta_{n+1}-\eta) = \frac1{\sqrt{|\T_n|}} N_{n+1}.\label{lienN}\end{equation}
In addition, the increasing process associated to $(N_k^{(n)})$ is given by
\begin{align*}
\langle N^{(n)}\rangle _k &= \frac1{|\T_n|} \sum_{i=1}^k \E\left[\left.\frac{v_{2i}^2}{d_i^2} \psi_i\psi_i^t \right| \calG_{i-1}\right],\\
	&= \frac1{|\T_n|} \sum_{i-1}^k \frac{(\mu_a^4-\sigma_a^4)X_i^4 + 4\sigma_a^2\sigma_c^2X_i^2 + (\mu_c^4-\sigma_c^4)}{d_i^2}\psi_i\psi_i^t \hspace{20pt} \text{a.s.}
\end{align*}
Consequently, we obtain from Lemma \ref{LFGN} that
$$\lim_{n\to\infty} \langle N^{(n)}\rangle _{t_n} = \E\left[\frac{(\mu_a^4-\sigma_a^4)T^4 + 4\sigma_a^2\sigma_c^2T^2 + (\mu_c^4-\sigma_c^4)}{(1+T^2)^4} \begin{pmatrix} T^4 & T^2 \\ T^2 & 1 \end{pmatrix} \right]=M_{ac} \hspace{10pt} \text{a.s.}$$
In order to verify Lyapunov's condition, let $\alpha>4$ be the constant in \H{H5} and let
$$\phi_n = \sum_{k=1}^{t_n} \E\left[\left.\|N_k^{(n)} - N_{k-1}^{(n)}\|^{\alpha/2} \right| \calG_{k-1}\right].$$
We clearly have
$$\|N_k^{(n)} - N_{k-1}^{(n)}\|^2 = \frac1{|\T_n|} \frac{v_{2k}^2}{d_k^2}\|\psi_k\|^2 \leq \frac1{|\T_n|} \frac{v_{2k}^2}{d_k},$$
which implies that
$$\|N_k^{(n)} - N_{k-1}^{(n)}\|^{\alpha/2} \leq \frac1{|\T_n|^{\alpha/4}} \frac{|v_{2k}|^{\alpha/2}}{d_k^{\alpha/4}}.$$
However, it exists a constant $\beta>0$ such that
\begin{equation}
|v_{2k}|^{\alpha/2} = |V_{2k}^2 - \sigma_a^2X_k^2 - \sigma_c^2|^{\alpha/2} \leq \beta(|V_{2k}|^{\alpha} + (\sigma_a^2X_k^2 + \sigma_c^2)^{\alpha/2}). \label{majv2k3}
\end{equation}
Moreover, we also have
$$|V_{2k}|^{\alpha} \leq \beta(|a_k-a|^{\alpha}|X_k|^{\alpha}+|\veps_{2k}-c|^{\alpha}).$$
Let 
$$Y=\max\(\sup_{n\geq0}\sup_{k\in\G_n} \E[|a_k-a|^\alpha|\calF_n],\sup_{n\geq0}\sup_{k\in\G_n} \E[|\veps_{2k}-c|^\alpha|\calF_n]\),$$
then it exists some constant $\gamma>0$  such that
$$\E[|V_{2k}|^{\alpha}|\calG_{k-1}]\leq \beta Y(1+|X_k|^{\alpha})\leq \gamma Y(1+X_k^2)^{\alpha/2}\as$$
This, together with \eqref{majv2k3}, ensures the existence of a constant $\delta>0$ such that
$$\E[|v_{2k}|^{\alpha/2}|\calG_{k-1}]\leq\delta Y(1+X_k^2)^{\alpha/2}\as$$
implying that
$$\E\left[\left.\|N_k^{(n)} - N_{k-1}^{(n)}\|^{\alpha/2} \right| \calG_{k-1}\right]\leq\frac{\delta Y}{|\T_n|^{\alpha/4}}\as$$
Then we can conclude that
$$\phi_n\leq \frac{\delta Y}{|\T_n|^{\alpha/4-1}}\as$$
which immediately leads, since $Y<\infty$ a.s., to
$$\lim_{n\to\infty} \phi_n = 0 \as$$
\noindent Therefore, Lyapunov's condition is satisfied and we find from Theorem 2.1.9 of \cite{Duflo} and \eqref{lienN} that
\begin{equation}\label{limN}
\frac1{\sqrt{|\T_{n-1}|}}N_n \liml \calN(0,M_{ac}).
\end{equation}
Hence, we obtain from \eqref{CVQ}, \eqref{limN} and Slutsky's lemma that
$$\sqrt{|\T_{n-1}|}(\eta_n-\eta) \liml \calN(0,D^{-1}M_{ac} D^{-1}).$$
Finally, \eqref{vitesseeta} ensures that
$$\sqrt{|\T_{n-1}|}(\wh\eta_n-\eta) \liml \calN(0,D^{-1}M_{ac} D^{-1}).$$
The proof of \eqref{TCLetad} follows exactly the same lines. \hfill $\square$

{\bf Proof of Theorem \ref{TCL}, third part.} It remains to establish the asymptotic normality given by \eqref{TCLrho}. Denote by $H^{(n)} = (H_k^{(n)})$ the square integrable martingale defined as
\begin{equation*} \label{lienH}
H_k^{(n)} = \frac1{\sqrt{|\T_n|}} \sum_{i=1}^k \frac{w_{2i}}{d_i} \psi_i.
\end{equation*}
We clearly have 
$$H_{t_n}^{(n)} = \frac1{\sqrt{|\T_n|}} \sum_{i=1}^{t_n} \frac{w_{2i}}{d_i} \psi_i = \frac1{\sqrt{|\T_n|}}H_{n+1}.$$
Moreover, the increasing process of $(H_k^{(n)})$ is given by
$$\langle H^{(n)}\rangle _k = \frac1{|\T_n|} \sum_{i=1}^k\frac{\E[w_{2i}^2|\calG_{i-1}] \psi_i \psi_i^t}{d_i^2}.$$
In addition, we already saw in Section 3 that
\begin{equation*}
\E[w_{2k}^2|\calF_n] = (\nu_{ab}^2-\rho_{ab}^2) X_k^4 + (\sigma_a^2\sigma_d^2 + \sigma_b^2\sigma_c^2+2\rho_{ab}\rho_{cd})X_k^2 + (\nu_{cd}^2-\rho_{cd}^2) \as
\end{equation*}
Then, we deduce once again from Lemma \ref{LFGN} that
\begin{align*}
&\lim_{n\to\infty} \langle H^{(n)}\rangle _{t_n} \\
&\hspace{30pt}= \E\left[\frac{(\nu_{ab}^2-\rho_{ab}^2) T^4 + (\sigma_a^2\sigma_d^2 + \sigma_b^2\sigma_c^2+2\rho_{ab}\rho_{cd})T^2 + (\nu_{cd}^2-\rho_{cd}^2)}{(1+T^2)^4}\begin{pmatrix} T^4 & T^2 \\ T^2 & 1 \end{pmatrix} \right]\\
&\hspace{30pt}= H\as
\end{align*}
In order to verify Lyapunov's condition, denote, with $\alpha>4$ the constant in \H{H5},
$$\phi_n = \sum_{k=1}^{t_n} \E\left[\left.\|H_k^{(n)} - H_{k-1}^{(n)}\|^{\alpha/2}\right|\calG_{k-1}\right].$$
As in the previous proof, we clearly have that
$$\|H_k^{(n)} - H_{k-1}^{(n)}\|^{\alpha/2} \leq \frac1{|\T_n|^{\alpha/4}} \frac{|w_{2k}|^{\alpha/2}}{d_k^{\alpha/4}}.$$
We can observe that it exists some constants $\beta>0$ and $\gamma>0$ such that
\begin{align*}
|w_{2k}|^{\alpha/2} &= |V_{2k}V_{2k+1} - \rho_{ab} X_k^2 -\rho_{cd}|^{\alpha/2} \leq \(|V_{2k}V_{2k+1}| + |\rho_{ab}| X_k^2 + |\rho_{cd}|\)^{\alpha/2},\\
	&\leq \beta(|V_{2k}V_{2k+1}|^{\alpha/2}+(|\rho_{ab}|X_k^2+|\rho_{cd}|)^{\alpha/2}), \\
	&\leq\gamma(|V_{2k}|^{\alpha} + |V_{2k+1}|^{\alpha}+(|\rho_{ab}|X_k^2+|\rho_{cd}|)^{\alpha/2}) .
\end{align*}
Hence, in the same way as in the proof of the second part, we can prove that it exists a constant $\delta>0$ and a random variable $Y$ such that $Y<\infty$ a.s.~verifying
$$\E[|w_k|^{\alpha/2}|\calG_{k-1}] \leq \delta Y(1+X_k^2)^{\alpha/2} \as$$
which immediately leads to
$$\E[\|H_k^{(n)} - H_{k-1}^{(n)}\|^{\alpha/2} |\calG_{k-1}] \leq \frac{\delta Y}{|\T_n|^{\alpha/4}} \as$$
which ensures that
$$\phi_n \leq \frac{\delta Y}{|\T_n|^{\alpha/4-1}} \as$$
Then, we obviously have that
$$\lim_{n\to\infty} \phi_n =0 \hspace{20pt} \text{a.s.}$$
and we can conclude that
$$\frac1{\sqrt{|\T_{n-1}|}}H_{n} \liml \calN(0,H).$$
In other words
$$\sqrt{|\T_{n-1}|} (\nu_{n} -\nu) \liml \calN(0,D^{-1}HD^{-1}).$$
Finally, we find via \eqref{vitesserho} that
$$\sqrt{|\T_{n-1}|} (\wh\nu_{n} -\nu) \liml \calN(0,D^{-1}HD^{-1})$$
which achieves the proof of Theorem \ref{TCL}. \hfill $\square$\\

\section{Numerical simulations}

The goal of this section is to illustrate by simulations the main results of this paper. In order to keep this section brief, we shall only focus our attention on the asymptotic normality of the WLS estimator of the unknown parameter $\theta$. On the one hand the random coefficient sequence $(a_n,b_n)$ is chosen to be i.i.d~sharing the same distribution as $(X+Y,X+Z)$ where $X\sim \calN(0.5,0.4)$, $Y\sim\calN(0,0.3)$ and $Z\sim\calN(-0.2,0.4)$. Those parameters have been chosen in order to satisfy \H{H1}. On the other hand, the driven noise sequence $(\veps_{2n},\veps_{2n+1})$ is chosen to be i.i.d.~sharing the same distribution as $(U+V,U+W)$ where $U\sim \calE(1)$, $V\sim\calE(2)$ and $W\sim\calE(3)$ and $\calE(\lambda)$ stands for the exponential distribution with parameter $\lambda>0$. The histograms are made by computing 4000 times $\wh\theta_n$ with $n=13$, and the variances of the theoretical normal distributions, which are plotted with the red curve, have been estimated by a Monte-Carlo procedure. One can observe in Figure 2 that the WLS estimator $\wh\theta_n$ performs very well in the estimation of $\theta$.

\begin{figure}[h!]
\includegraphics[scale=0.65]{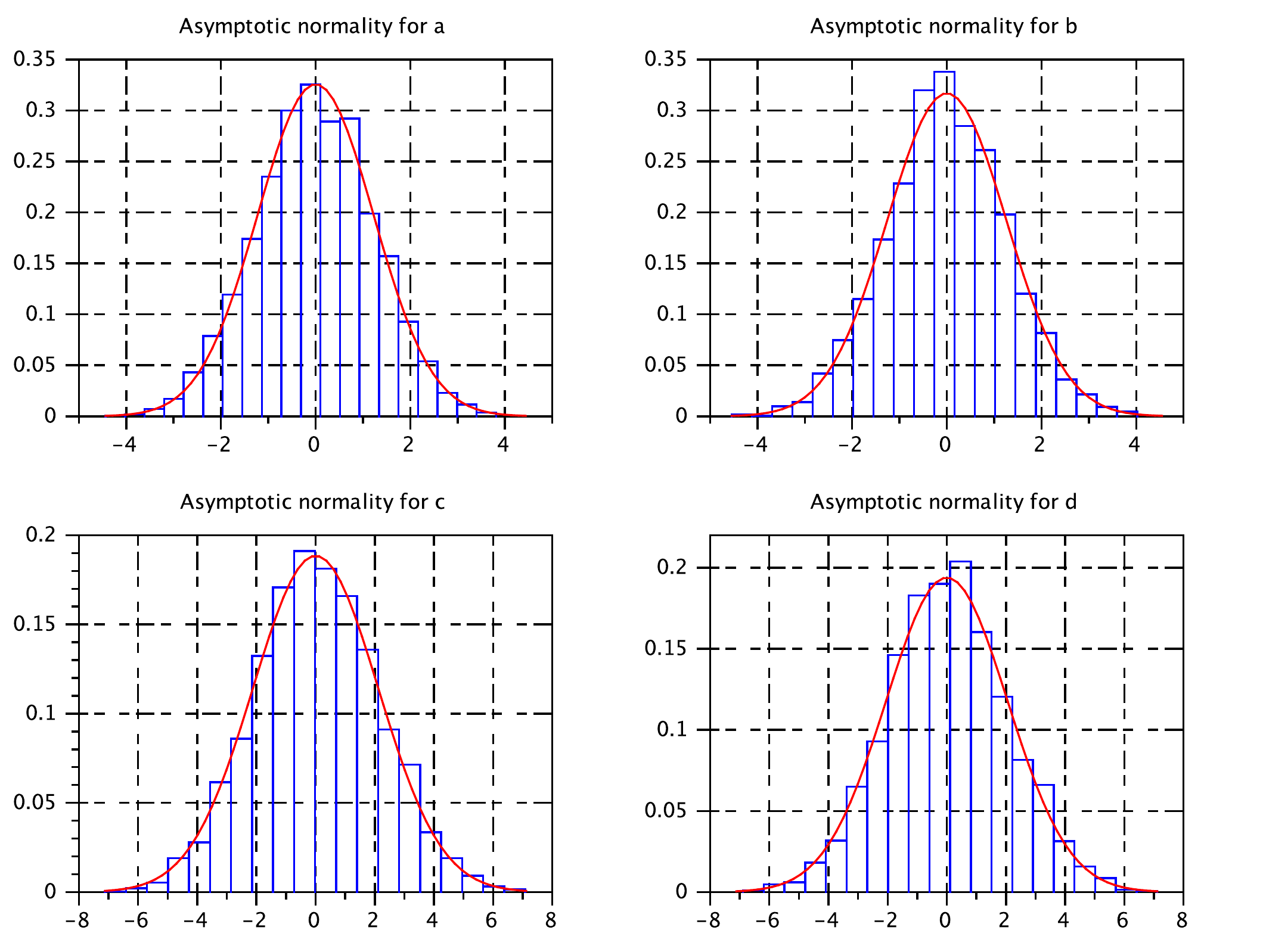}
\caption{Illustration of the asymptotic normalities of $a$, $b$, $c$ and $d$.}
\end{figure}
\hspace{0pt}\\
{\bf Acknowledgement.} I would like to thank Bernard Bercu for his helpful suggestions and for thorough readings of the paper.
\nocite{*}
\bibliographystyle{acm}
\bibliography{RCBARv1}

\end{document}